\documentclass{article}
\usepackage[cp1251]{inputenc}
\usepackage[russian]{babel}
\usepackage{amssymb,amsfonts,amsmath,amsthm,amscd}
\usepackage{graphicx}
\usepackage{enumerate}
\usepackage{soul}
\usepackage[matrix,arrow,curve]{xy}
\usepackage{longtable}
\usepackage{array}
\usepackage{bm}
\RequirePackage{russlh}
\RequirePackage{mathlh}

\newdimen\symskip
\newdimen\defskip
\newdimen\parind
\parind=\parindent
\newdimen\leftmarge
\newdimen\theoremshape
\topsep\defskip
\makeatletter
\newcommand*{\клей}{\nobreak\hskip\z@skip}
\newcommand{\?}{\,\nobreak\hskip0pt}
\renewcommand{\"}{''}
\renewcommand{\:}{\textup{:}}
\renewcommand{\~}{\textup{;}}
\DeclareRobustCommand*{\т}{~\textemdash{} }
\DeclareRobustCommand*{\д}{\клей\hbox{-}\клей}
\newcommand{\no}{}
\renewcommand{\@listI}{\settowidth\labelwidth{\labheadi{\no}}\listipar{\parind}{\labelwidth}}
\newcommand{\listivpar}{\topsep\defskip\partopsep0pt\parsep-\parskip\itemsep0.5\topsep}
\newcommand{\listipar}[2]{\rightmargin0pt\leftmargin#1\labelsep#1\advance\labelsep-#2\itemindent0pt\listivpar}
\renewcommand{\@listii}{\settowidth\labelwidth{\labheadii{\@roman{\no}}}\listiipar{\parind}{\labelwidth}}
\newcommand{\listiivpar}{\topsep0.5\defskip\partopsep0pt\parsep-\parskip\itemsep0.5\topsep}
\newcommand{\listiipar}[2]{\rightmargin0pt\leftmargin#1\labelsep#1\advance\labelsep-#2\itemindent0pt\listiivpar}
\def\thempfn{\ifcase\value{footnote}1\or *\or **\or ***\else\@ctrerr\fi}
\renewcommand\footnoterule{%
  \kern-3\p@
  \hrule\@width1in
  \kern2.6\p@}
\makeatother

\makeatletter

\renewcommand{\@biblabel}[1]{[#1]}
\renewenvironment{thebibliography}[1]
     {\renewcommand{\refname}{Литература}%
      \section*{\refname}%
      \@mkboth{\MakeUppercase\refname}{\MakeUppercase\refname}%
      \list{\@biblabel{\@arabic\c@enumiv}}%
           {\itemsep\baselineskip
            \leftmargin\parind
            \settowidth\labelwidth{\@biblabel{#1}}%
            \labelsep\parind\advance\labelsep-\labelwidth
            \@openbib@code
            \usecounter{enumiv}%
            \let\p@enumiv\@empty
            \renewcommand\theenumiv{\@arabic\c@enumiv}}%
      \sloppy
      \clubpenalty4000
      \@clubpenalty\clubpenalty
      \widowpenalty4000%
      \sfcode`\.\@m}
     {\def\@noitemerr
       {\@latex@warning{Empty `thebibliography' environment}}%
      \endlist}

\def\@maketitle{%
  \newpage
  \vskip0.5em%
  УДК \udk%
  \vskip0.5em%
  MSC \msc%
  \vskip1em%
  \begin{center}\bf%
  \let\footnote\thanks%
   {\Large\@author\par}%
   \vskip1.5em%
   {\LARGE\@title\par}%
   \vskip1em%
   {\large\@date}%
  \end{center}%
  \par
  \vskip1.5em}

\def\@title{\@latex@warning@no@line{No \noexpand\title given}}

\makeatother

\makeatletter

\renewcommand\sectionmark[1]{%
 \markright{%
  \ifnum \c@secnumdepth >\z@
   \thesection. \ %
  \fi
 #1}}%

\renewcommand{\section}{\@startsection{section}{1}{0pt}%
{5.5ex plus .5ex minus .2ex}{1.5ex plus .3ex}%
{\center\normalfont\Large\bfseries\sffamily\bom}}
\renewcommand{\subsection}{\@startsection{subsection}{2}{0pt}%
{4.5ex plus .4ex minus .2ex}{0.75ex plus .2ex}%
{\center\normalfont\large\bfseries\sffamily\bom}}
\renewcommand{\subsubsection}{\@startsection{subsubsection}{3}{0pt}%
{2.5ex plus .5ex minus .2ex}{1ex plus .2ex}%
{\center\normalfont\bfseries\sffamily\bom}}

\newcommand{\Ss}{\textup{\S\,}}

\def\@postskip@{\hskip.5em\relax}

\def\postsection{.\@postskip@}

\def\postsubsection{.\@postskip@}

\def\postsubsubsection{.\@postskip@}

\def\postparagraph{.\@postskip@}

\def\postsubparagraph{.\@postskip@}

\def\@seccntformat#1{\csname pre#1\endcsname\csname the#1\endcsname\csname post#1\endcsname}

\makeatother

\renewcommand{\thesection}{\textup{\arabic{section}}}

\newcommand{\parr}{\par\addvspace{\defskip}}
\newcommand{\theo}[2]{\newtheorem{#1}{#2}[section]}
\newcommand{\deff}[2]{\newenvironment{#1}{\parr\textbf{#2.}}{\parr}}
\theo{cas}{Случай}
\theo{problem}{Проблема}
\theo{theorem}{Теорема}
\theo{lemma}{Лемма}
\theo{prop}{Предложение}
\theo{stm}{Утверждение}
\theo{imp}{Следствие}
\theo{ex}{Пример}
\deff{df}{Определение}
\deff{note}{Замечание}
\deff{denote}{Обозначение}
\deff{denotes}{Обозначения}
\deff{hint}{Указание}
\deff{answer}{Ответ\:}
\newenvironment{cass}[1]{\begin{cas}\label{#1}\upshape}{\end{cas}}

\makeatletter
\def\@begintheorem#1#2[#3]{%
  \deferred@thm@head{\the\thm@headfont \thm@indent
    \@ifempty{#1}{\let\thmname\@gobble}{\let\thmname\@iden}%
    \@ifempty{#2}{\let\thmnumber\@gobble}{\let\thmnumber\@iden}%
    \@ifempty{#3}{\let\thmnote\@gobble}{\let\thmnote\@iden}%
    \thm@notefont{\bfseries\upshape}%
    \indent%
    \thm@swap\swappedhead\thmhead{#1}{#2}{#3}%
    \the\thm@headpunct
    \thmheadnl 
    \hskip\thm@headsep
  }%
  \ignorespaces}
\renewenvironment{proof}{\setcounter{cas}{0}\parr\pushQED{\qed}\normalfont$\square\quad$}{\setcounter{cas}{0}\popQED\@endpefalse\parr}

\makeatother

\newcommand{\labheadi}[1]{\textup{#1)}}
\newcommand{\labheadii}[1]{\textup{(#1)}}

\newenvironment{nums}[1]{\renewcommand{\no}{#1}\begin{enumerate}}{\end{enumerate}}
\newcommand{\eqn}[1]{\begin{equation}#1\end{equation}}
\newcommand{\equ}[1]{\begin{equation*}#1\end{equation*}}

\newcommand{\case}[1]{\begin{cases}#1\end{cases}}

\newcommand{\rbmat}[1]{\begin{pmatrix}#1\end{pmatrix}}

\makeatletter

\def\LT@makecaption#1#2#3{%
  \LT@mcol\LT@cols c{\hbox to\z@{\hss\parbox[t]\LTcapwidth{%
    \sbox\@tempboxa{#1{#2. }#3}%
    \ifdim\wd\@tempboxa>\hsize
      #1{#2. }#3%
    \else
      \hbox to\hsize{\hfil\box\@tempboxa\hfil}%
    \fi
    \endgraf\vskip\baselineskip}%
  \hss}}}

\@addtoreset{equation}{section}
\@addtoreset{footnote}{section}

\newcounter{numt}
\newcounter{col}
\newcounter{coll}

\makeatother

\renewcommand{\ge}{\geqslant}
\renewcommand{\le}{\leqslant}
\newcommand{\fa}{\,\forall\,}
\newcommand{\exi}{\,\exists\,}

\newcommand{\bes}{\infty}
\newcommand{\es}{\varnothing}
\newcommand{\eqi}{\equiv}
\newcommand{\subs}{\subset}
\newcommand{\sups}{\supset}

\newcommand{\sm}{\setminus}
\newcommand{\wo}{\backslash}
\newcommand{\cln}{\colon}

\newcommand{\wg}{\wedge}

\newcommand{\Ra}{\Rightarrow}

\newcommand{\os}[1]{\overset{#1}}
\newcommand{\Inn}[1]{\smash{\os{\circ}{\smash{#1}\vph{^{_{^{_{c}}}}}}}\vph{#1}}
\newcommand{\wt}{\widetilde}

\newcommand{\nd}{\,\&\,}
\newcommand{\sst}[1]{\substack{#1}}

\newcommand{\suml}[2]{\sum\limits_{{#1}}^{{#2}}}
\newcommand{\sums}[1]{\sum\limits_{{#1}}}

\newcommand{\opluss}[1]{\bigoplus\limits_{{#1}}^{}}

\newcommand{\capl}[2]{\bigcap\limits_{{#1}}^{{#2}}}
\renewcommand{\caps}[1]{\bigcap\limits_{{#1}}}
\newcommand{\cupl}[2]{\bigcup\limits_{{#1}}^{{#2}}}

\newcommand*{\bw}[1]{#1\nobreak\discretionary{}{\hbox{$\mathsurround=0pt #1$}}{}}
\newcommand{\sco}{,\ldots,}
\newcommand{\spl}{\bw+\ldots\bw+}
\newcommand{\smn}{\bw-\ldots\bw-}
\newcommand{\seq}{\bw=\ldots\bw=}

\newcommand{\sge}{\bw\ge\ldots\bw\ge}

\newcommand{\sd}{\bw\cdot\ldots\bw\cdot}

\newcommand{\ha}[1]{\left\langle#1\right\rangle}
\newcommand{\ba}[1]{\bigl\langle#1\bigr\rangle}

\newcommand{\br}[1]{\bigl(#1\bigr)}
\newcommand{\Br}[1]{\Bigl(#1\Bigr)}

\newcommand{\ter}[1]{\textup{(}#1\textup{)}}
\newcommand{\bgm}[1]{\bigl|#1\bigr|}
\newcommand{\Bm}[1]{\Bigl|#1\Bigr|}
\newcommand{\hn}[1]{\left\|#1\right\|}
\newcommand{\bn}[1]{\bigl\|#1\bigr\|}
\newcommand{\Bn}[1]{\Bigl\|#1\Bigr\|}

\newcommand{\bc}[1]{\bigl\{#1\bigr\}}
\newcommand{\BC}[1]{\Bigl\{#1\Bigr\}}

\newcommand{\bbl}{\bigm\wo}

\newcommand{\mbb}{\mathbb}
\newcommand{\mbf}{\mathbf}
\newcommand{\mcl}{\mathcal}
\newcommand{\mfr}{\mathfrak}

\newcommand{\R}{\mbb{R}}

\newcommand{\Z}{\mbb{Z}}
\newcommand{\N}{\mbb{N}}

\newcommand{\Cbb}{\mbb{C}}

\newcommand{\Eb}{\mbb{E}}

\newcommand{\Pc}{\mcl{P}}

\newcommand{\agt}{\mfr{a}}

\newcommand{\ggt}{\mfr{g}}
\newcommand{\hgt}{\mfr{h}}

\newcommand{\tgt}{\mfr{t}}
\newcommand{\pd}{\partial}

\newcommand{\al}{\alpha}

\newcommand{\de}{\delta}
\newcommand{\De}{\Delta}
\newcommand{\ep}{\varepsilon}
\newcommand{\la}{\lambda}
\newcommand{\La}{\Lambda}

\newcommand{\si}{\sigma}

\newcommand{\ta}{\theta}
\newcommand{\ph}{\varphi}
\newcommand{\om}{\omega}
\newcommand{\Om}{\Omega}

\DeclareMathOperator{\Aut}{Aut}
\DeclareMathOperator{\Lie}{Lie}
\DeclareMathOperator{\Ker}{Ker}

\DeclareMathOperator{\ad}{ad}

\DeclareMathOperator{\rk}{rk}

\DeclareMathOperator{\id}{id}
\DeclareMathOperator{\conv}{conv}

\newcommand{\GL}{\mbf{GL}}
\newcommand{\SL}{\mbf{SL}}
\newcommand{\Or}{\mbf{O}}
\newcommand{\SO}{\mbf{SO}}

\newcommand{\Un}{\mbf{U}}
\newcommand{\SU}{\mbf{SU}}
\newcommand{\glg}{\mfr{gl}}
\newcommand{\slg}{\mfr{sl}}

\newcommand{\sug}{\mfr{su}}

\newcommand{\bom}{\boldmath}

\newcommand{\vph}[1]{\vphantom{#1}}

\begin{document}

\author{О.\,Г.\?Стырт}
\title{О~пространстве орбит\\
неприводимого представления\\
специальной унитарной группы}
\date{}
\newcommand{\udk}{512.815.1+512.815.6+512.816.1+512.816.2}
\newcommand{\msc}{22E46+17B10+17B20+17B45}

\maketitle

{\leftskip\parind\rightskip\parind
Доказано, что фактор неприводимого представления специальной унитарной группы ранга более~$1$ не может быть гладким многообразием.


\smallskip

\textbf{Ключевые слова\:} группа Ли, топологический фактор действия.

\smallskip

It is proved that the factor of an irreducible representation of the special unitary group of rank greater than~$1$ can not be a~smooth manifold.

\smallskip

\textbf{Key words\:} Lie group, topological factor of an action.\par}

\section{Введение}\label{introd}

Настоящая работа является непосредственным продолжением статей \cite{My1,My2}. Вначале дадим три базовых определения, игравших ключевую роль
и~в~указанных работах.

\begin{df} Непрерывное отображение гладких многообразий назовём \textit{ку\-соч\-но-глад\-ким}, если оно переводит любое гладкое подмногообразие
в~конечное объединение гладких подмногообразий.
\end{df}

В частности, всякое собственное гладкое отображение гладких многообразий является ку\-соч\-но-глад\-ким.

Рассмотрим дифференцируемое действие компактной группы Ли~$G$ на гладком многообразии~$M$.

\begin{df} Будем говорить, что фактор $M/G$ \textit{ку\-соч\-но-диф\-фео\-мор\-фен} гладкому многообразию~$M'$, если топологический фактор $M/G$
гомеоморфен~$M'$, причём отображение факторизации $M\to M'$ ку\-соч\-но-гладкое.
\end{df}

\begin{df} Будем говорить, что фактор $M/G$ является \textit{гладким многообразием}, если он ку\-соч\-но-диф\-фео\-мор\-фен некоторому гладкому
многообразию.
\end{df}

Перейдём непосредственно к~постановке задачи.

Пусть $V$\т вещественное векторное пространство, а~$G\subs\GL(V)$\т компактная линейная группа. Как и~в~работах~\cite{My1,My2}, нас интересует вопрос
о~том, является ли фактор $V/G$ топологическим многообразием, а~также является ли он гладким многообразием. Следуя~\cite{My1,My2}, будем называть
топологическое многообразие просто <<многообразием>>.

Обозначим через~$V_{\Cbb}$ комплексное пространство $V\otimes\Cbb$, через~$\ggt$\т линейную алгебру Ли $\Lie G\subs\glg(V)$, а~через~$\ggt_{\Cbb}$\т
комплексную линейную алгебру Ли $\ggt\otimes\Cbb\subs\glg(V_{\Cbb})$.

К~настоящему моменту разобраны два случая\: $[\ggt,\ggt]=0$ (см.~\cite{My1}) и~$\ggt\cong\sug_2$ (см.~\cite{My2}). В~данной же работе рассматривается
случай, когда $\ggt\cong\sug_{r+1}$, где $r=\rk\ggt>1$, причём линейная алгебра Ли $\ggt\subs\glg(V)$ неприводима.

Будем обозначать представление комплексной редуктивной группы Ли, сопряжённое представлению~$R$, через~$R'$.

Говоря о~неразложимых системах простых корней, мы будем использовать нумерацию простых корней, принятую в~\cite[табл.\,1]{VO} и~\cite[табл.\,1]{Elsh},
обозначая через $(i_1\sco i_m)$ подмножество простых корней с~номерами $i_1\sco i_m$, через~$\ph_i$\т фундаментальный вес с~номером~$i$, а~через
$\ph(\hgt)$\т линейную алгебру, отвечающую неприводимому представлению комплексной простой алгебры~$\hgt$ со старшим весом~$\ph$.

Пространство~$V$ обладает $G$\д инвариантным скалярным умножением и~поэтому может (и~будет) рассматриваться как евклидово пространство, на котором
группа~$G$ действует ортогональными операторами. Таким образом, $G\subs\Or(V)$.

Пусть $\wt{R}$\т тавтологическое представление $\ggt_{\Cbb}\cln V_{\Cbb}$. Возможны следующие случаи\:
\begin{nums}{-1}
\item $\wt{R}=R$, где $R$\т точное неприводимое представление\~
\item $\wt{R}=R+R'$, где $R$\т точное неприводимое представление.
\end{nums}
Во втором случае пространство~$V$ обладает $\ggt$\д инвариантной комплексной структурой, в~результате чего естественным образом возникает (комплексное)
представление $\ggt_{\Cbb}\cln V$, изоморфное представлению~$R$.

Теперь предположим, что $\ggt\cong\sug_{r+1}$, где $r=\rk\ggt>1$, причём линейная алгебра Ли $\ggt\subs\glg(V)$ неприводима.

В~настоящей работе будут доказаны теоремы \ref{main} и~\ref{main1}.

\begin{theorem}\label{main} Фактор $V/G$ может быть гладким многообразием лишь в~следующих случаях\:
\begin{nums}{-1}
\item представление~$\wt{R}$ алгебры~$\ggt_{\Cbb}$ совпадает~с~представлением~$R$ и~изоморфно одному из представлений $\ad$, $R_{\ph_2}$ \ter{$r=3$},
$R_{2\ph_2}$ \ter{$r=3$} и~$R_{\ph_4}$ \ter{$r=7$}\~
\item $\wt{R}=R+R'$, а~представление~$R$ изоморфно \ter{с~точностью до внешнего автоморфизма алгебры~$\ggt_{\Cbb}$} одному из представлений
$R_{\ph_1}$, $R_{2\ph_1}$, $R_{\ph_2}$ \ter{$r>3$} и~$R_{\ph_3}$ \ter{$r=5$}.
\end{nums}
\end{theorem}

\begin{theorem}\label{main1} Если $G=G^0$, а~представление~$R$ изоморфно \ter{с~точностью до внешнего автоморфизма алгебры~$\ggt_{\Cbb}$} одному из
представлений $\ad$, $R_{\ph_1}$, $R_{2\ph_1}$, $R_{\ph_2}$ \ter{$r>2$}, $R_{2\ph_2}$ \ter{$r=3$}, $R_{\ph_3}$ \ter{$r=5$} и~$R_{\ph_4}$ \ter{$r=7$}, то
$V/G$\т не многообразие.
\end{theorem}

\begin{imp}\label{main2} При $G=G^0$ фактор $V/G$ не является гладким многообразием.
\end{imp}

В~\Ss\ref{facts} введён ряд используемых в~работе обозначений, а~также доказаны некоторые вспомогательные утверждения. В~\Ss\ref{promain} доказана
теорема~\ref{main}, а~в~\Ss\ref{promain1}\т теорема~\ref{main1}.

Автор благодарит профессора Э.\,Б.\?Винберга за постоянную поддержку в~научной деятельности и~многочисленные ценные советы.

\section{Обозначения и вспомогательные факты}\label{facts}

В~этом параграфе приведён ряд вспомогательных обозначений и~утверждений, в~том числе заимствованных из~\cite{My1,My2,My0} (все новые утверждения\т
с~доказательствами).

Для линейного представления группы Ли~$G$ (соотв. алгебры Ли~$\ggt$) в~пространстве~$V$ стабилизатор (соотв. стационарная подалгебра) вектора $v\in V$
будет стандартным образом обозначаться через~$G_v$ (соотв. через~$\ggt_v$).

В~работе~\cite{My0} для каждой неразложимой системы простых корней~$\Pi$ некоторым образом определялось подмножество $\pd\Pi\subs\Pi$.

Все неразложимые системы простых корней~$\Pi$, для которых $\pd\Pi\ne\Pi$, перечислены в~\cite[\Ss4,~табл.\,1]{My0} с~указанием подмножества
$\pd\Pi\subs\Pi$. Так, если $\Pi\cong A_r$, то
\equ{
\pd\Pi=\case{
(1,2,r-1,r),&r\ge6;\\
\Pi,&r<6.}}

\subsection{Представления компактных групп}

Допустим, что имеется евклидово пространство~$V$, компактная группа Ли~$G$ с~касательной алгеброй~$\ggt$, линейное представление $G\to\Or(V)$ и~его
дифференциал\т представление $\ggt\cln V$.

Для произвольного вектора $v\in V$ обозначим через~$N_v$ подпространство $(\ggt v)^{\perp}\subs V$. Очевидно, что $\ggt_v=\Lie G_v$ и~$G_vN_v=N_v$
($v\in V$).

\begin{lemma}\label{slice} Если $V/G$\т \ter{гладкое} многообразие, то каждый из факторов $N_v/G_v$, $v\in V$, также является \ter{гладким}
многообразием.
\end{lemma}

\begin{proof} См. лемму~2.3 в~\cite[\Ss2]{My2}.
\end{proof}

\begin{lemma}\label{xiV} Допустим, что $\dim\ggt=1$.
\begin{nums}{-1}
\item Если $V/G$\т многообразие, то для любого вектора $\xi\in\ggt$ имеем $\dim(\xi V)\ne2$.
\item Если $V/G$\т гладкое многообразие, то для любого вектора $\xi\in\ggt$ имеем $\dim(\xi V)\le6$.
\end{nums}
\end{lemma}

\begin{proof} См. следствие~2.3 в~\cite[\Ss2]{My2}.
\end{proof}

\begin{lemma}\label{qV} Если $\ggt\cong\sug_2$, а~$V/G$\т гладкое многообразие, то сумма целых частей половин размерностей всех неприводимых компонент
представления $\ggt\cln V$ \ter{с~учётом кратностей} не превосходит~$4$.
\end{lemma}

\begin{proof} См. теорему~1.1 в~\cite[\Ss1]{My2}.
\end{proof}

\begin{imp}\label{xiV3} Если $\ggt\cong\sug_2$, а~$V/G$\т гладкое многообразие, то для всякого $\xi\in\ggt$ имеем $\dim(\xi V)\le8$.
\end{imp}

\begin{imp}\label{rk1} Если $\rk\ggt=1$, а~$V/G$\т гладкое многообразие, то для всякого $\xi\in\ggt$ имеем $\dim(\xi V)\le\dim\ggt+5$.
\end{imp}

\begin{proof} Вытекает из леммы~\ref{xiV} и~следствия~\ref{xiV3}.
\end{proof}

\begin{imp}\label{rk1x} Если $\rk\ggt=1$, а~$V/G$\т гладкое многообразие, то любой ненулевой вектор $\xi\in\ggt$ удовлетворяет неравенству
$\dim(\xi V)\le\dim[\xi,\ggt]+6$.
\end{imp}

\begin{proof} Поскольку $\rk\ggt=1$, для всякого ненулевого вектора $\xi\in\ggt$ имеем $\Ker(\ad\xi)=\R\xi$, откуда $\dim[\xi,\ggt]=\dim\ggt-1$.
Осталось воспользоваться следствием~\ref{rk1}.
\end{proof}

\begin{lemma}\label{didi} Предположим, что $V/G$\т гладкое многообразие. Для любого вектора $v\in V$, удовлетворяющего равенству $\rk\ggt_v=1$, и~для
любого ненулевого вектора $\xi\in\ggt_v$ имеем $\dim(\xi V)\le\dim[\xi,\ggt]+6$.
\end{lemma}

\begin{proof} Согласно лемме~\ref{slice}, $N_v/G_v$\т гладкое многообразие. Применяя к~представлению $G_v\cln N_v$ следствие~\ref{rk1x}, получаем, что
$\dim(\xi N_v)\le\dim[\xi,\ggt_v]+6$, и, таким образом,
$\dim(\xi V)=\dim\br{\xi(\ggt v)}+\dim(\xi N_v)\le\dim\br{\xi(\ggt v)}+\dim[\xi,\ggt_v]+6=\dim[\xi,\ggt]+6$.
\end{proof}

\begin{lemma}\label{dine} Предположим, что $V/G$\т многообразие. Для любого вектора $v\in V$, такого что $\dim\ggt_v=1$, и~для любого вектора
$\xi\in\ggt_v$ имеем $\dim(\xi V)\ne\dim[\xi,\ggt]+2$.
\end{lemma}

\begin{proof} Согласно лемме~\ref{slice}, $N_v/G_v$\т многообразие. Применяя к~представлению $G_v\cln N_v$ лемму~\ref{xiV}, получаем, что
$\dim(\xi N_v)\ne2$. Далее, $[\xi,\ggt_v]=0$, $\dim\br{\xi(\ggt v)}=\dim[\xi,\ggt]$, откуда
$\dim(\xi V)=\dim\br{\xi(\ggt v)}+\dim(\xi N_v)=\dim[\xi,\ggt]+\dim(\xi N_v)\ne\dim[\xi,\ggt]+2$.
\end{proof}

\begin{lemma}\label{orth} Пусть $v_1\sco v_k\in V$\т произвольные векторы. Если $(\ggt v_i,v_j)=0$ для любых $i,j\in\{1\sco k\}$, $i\ne j$, то
в~подпространстве $\ha{v_1\sco v_k}\subs V$ найдётся вектор со стабилизатором $\capl{i=1}{k}G_{v_i}\subs G$.
\end{lemma}

\begin{proof} Применим индукцию по числу $k\in\N$.

При $k=1$ доказывать нечего.

Докажем утверждение леммы для числа $k\in\N\sm\{1\}$ в~предположении, что для числа $k-1\in\N$ оно уже доказано.

По предположению индукции найдётся вектор $v\in\ha{v_1\sco v_{k-1}}$ со стабилизатором $G_v=\capl{i=1}{k-1}G_{v_i}\subs G$. Согласно условию,
$(\ggt v_i,v_k)=0$ ($i=1\sco k-1$), откуда $(\ggt v,v_k)=0$, $v_k\in N_v$. Значит, для некоторого вектора $v'=v+\ep v_k\in\ha{v_1\sco v_k}\cap N_v$,
$\ep\in\R_{>0}$, имеем $G_{v'}\subs G_v$ и, как следствие, $G_{v'}=G_v\cap G_{\ep v_k}=\capl{i=1}{k}G_{v_i}$.
\end{proof}

\begin{stm}\label{read} Если $G=\SU_m\times\SU_m$ и~$V=\R^d\oplus\glg_m(\Cbb)$ \ter{$d\ge0$, $m>1$}, причём действие $G\cln V$ осуществляется по правилу
$(g_1,g_2)\cln(x,y)\to(x,g_1yg_2^{-1})$ \ter{$g_1,g_2\in\SU_m$, $x\in\R^d$, $y\in\glg_m(\Cbb)$}, то $V/G$\т не многообразие.
\end{stm}

\begin{proof} Положим $v:=E\in\glg_m(\Cbb)\subs V$. Легко видеть, что $G_v=\bc{(g,g)\cln g\in\SU_m}\subs G$, $\ggt v=\sug_m\subs\glg_m(\Cbb)\subs V$
и~$V=\R^d\oplus(\Cbb E)\oplus\sug_m\oplus(i\cdot\sug_m)$. Отсюда следует, что (с~учётом естественного изоморфизма групп $\SU_m$ и~$G_v$) представление
$G_v\cln N_v$ изоморфно прямой сумме тождественного действия $G_v\cln\R^{d+2}$ и~присоединённого представления $\SU_m\cln\sug_m$. Значит,
$N_v/G_v\cong\R^{d+2}\times\R_{\ge0}^{m-1}\cong\R_{\ge0}^{m+d+1}$. Таким образом, фактор $N_v/G_v$ не является многообразием. Согласно лемме~\ref{slice},
не является им и~фактор $V/G$.
\end{proof}

Предположим, что представление $\ggt\cln V$ неприводимо.

Обозначим через~$\ggt_{\Cbb}$ комплексную редуктивную алгебру Ли $\ggt\otimes\Cbb$.

Как легко видеть, найдутся комплексное пространство~$\wt{V}$ и~неприводимое представление $\ggt_{\Cbb}\cln\wt{V}$, такие что представление
$\ggt\cln V$ является овеществлением либо вещественной формой представления $\ggt\cln\wt{V}$.

Положим $\de:=1\in\R$ при $\wt{V}=V\otimes\Cbb$ и~$\de:=2\in\R$ при $\wt{V}=V$.

Фиксируем максимальную коммутативную подалгебру~$\tgt$ алгебры~$\ggt$ и~картановскую подалгебру $\tgt_{\Cbb}:=\tgt\otimes\Cbb$ алгебры~$\ggt_{\Cbb}$.
В~результате возникают система корней $\De\subs\tgt_{\Cbb}^*$ и~её группа Вейля $W\subs\GL(\tgt_{\Cbb}^*)$. Фиксируем систему простых корней
$\Pi\subs\De\subs\tgt_{\Cbb}^*$.

Пусть $\Inn{\La}\subs\tgt_{\Cbb}^*$\т множество весов неприводимого представления $\ggt_{\Cbb}\cln\wt{V}$, $\la\in\tgt_{\Cbb}^*\sm\{0\}$\т старший вес
данного представления относительно системы простых корней $\Pi\subs\De\subs\tgt_{\Cbb}^*$, а~$\La$\т подмножество
$W\la\subs\Inn{\La}\subs\tgt_{\Cbb}^*$. Для комплексного подпространства $\agt\subs\tgt_{\Cbb}$ положим
$\ggt_{\Cbb}(\agt):=\agt\oplus\Br{\opluss{\sst{\al\in\De,\\h_{\al}\in\agt}}(\ggt_{\Cbb})_{\al}}\subs\ggt_{\Cbb}$. Наконец, для подмножества
$\Om\subs\tgt_{\Cbb}^*$ введём обозначения $\tgt_{\Cbb}^{\Om}:=\caps{\om\in\Om}(\Ker\om)\subs\tgt_{\Cbb}$
и~$\ggt^{\Om}:=\ggt_{\Cbb}(\tgt_{\Cbb}^{\Om})\cap\ggt\subs\ggt$. Имеем $\ggt^{\Om_1\cup\Om_2}=\ggt^{\Om_1}\cap\ggt^{\Om_2}\subs\ggt$
($\Om_1,\Om_2\subs\tgt_{\Cbb}^*$).

Рассмотрим произвольный вес $\la'\in\La$.

Для всякого вектора $v\ne0$ (нетривиального) подпространства $\wt{V}_{\la'}\subs\wt{V}$ в~алгебре~$\ggt_{\Cbb}$ подалгебра~$(\ggt_{\Cbb})_v$ совпадает
с~прямой суммой всех подпространств $\Ker\la'\subs\tgt_{\Cbb}$ и~$(\ggt_{\Cbb})_{\al}$ ($\al\in\De$, $\ha{\la'|\al}\ge0$), откуда
$\ggt_v=\ggt^{\{\la'\}}\subs\ggt$.

Если $\wt{V}=V\otimes\Cbb$ и~$2\la'\notin\De\cup(\De+\De)$, то $-\la'\in\La$ и, кроме того,
$(\ggt_{\Cbb}\wt{V}_{\la'})\cap(\ggt_{\Cbb}\wt{V}_{-\la'})=0$, $(\ggt\wt{V}_{\la'})\cap(\ggt\wt{V}_{-\la'})=0$, вследствие чего для всякого ненулевого
вектора~$v$ (нетривиального) подпространства $(\wt{V}_{\la'}\oplus\wt{V}_{-\la'})\cap V\subs V$ имеем $\ggt_v=\ggt^{\{\la'\}}\subs\ggt$.

Из вышесказанного, а~также из леммы~\ref{orth} вытекает следующая лемма.

\begin{lemma}\label{sts} Пусть $\Om\subs\La$\т некоторое подмножество, удовлетворяющее равенству $(\Om-\Om)\cap\De=\es$. Предположим, что выполнено одно
из следующих условий\:
\begin{nums}{-1}
\item $\de=2$\~
\item $\de=1$, $2\la\notin\De\cup(\De+\De)$ и~$(\Om+\Om)\cap\De=\es$.
\end{nums}
Тогда найдётся вектор $v\in V$, такой что $\ggt_v=\ggt^{\Om}\subs\ggt$.
\end{lemma}

Введём обозначение~$\hn{\cdot}$ для порядка произвольного подмножества множества $\Inn{\La}\subs\tgt_{\Cbb}^*$ весов представления
$\ggt_{\Cbb}\cln\wt{V}$ \textit{с~учётом кратностей весов}.

\begin{lemma}\label{bms} Предположим, что $V/G$\т гладкое многообразие. Если $v\in V$, $\rk\ggt_v=1$ и~$\xi\in(\ggt_v\cap\tgt)\sm\{0\}$, то
$\de\cdot\Bn{\bc{\la'\in\Inn{\La}\cln\la'(\xi)\ne0}}\le\Bm{\bc{\al\in\De\cln\al(\xi)\ne0}}+6$.
\end{lemma}

\begin{proof} В~силу леммы~\ref{didi}, $\de\cdot\dim_{\Cbb}(\xi\wt{V})=\dim(\xi V)\le\dim[\xi,\ggt]+6=\dim_{\Cbb}[\xi,\ggt_{\Cbb}]+6$, что влечёт
требуемое.
\end{proof}

\begin{imp}\label{nosm} Пусть $\Om\subs\La$\т произвольное подмножество. Обозначим через~$H$ подпространство $\ha{\Om}_{\Cbb}\subs\tgt_{\Cbb}^*$. Если
\begin{gather*}
\dim_{\Cbb}(\tgt_{\Cbb}^*/H)=1;\quad\quad\de\cdot\bn{\Inn{\La}\sm H}>|\De\sm H|+6;\quad\quad(\Om-\Om)\cap\De=\es;\\
(\de=1)\quad\Ra\quad\Br{\br{2\la\notin\De\cup(\De+\De)}\nd\br{(\Om+\Om)\cap\De=\es}},
\end{gather*}
то фактор $V/G$ не является гладким многообразием.
\end{imp}

\begin{proof} Вытекает из лемм \ref{sts} и~\ref{bms}.
\end{proof}

\begin{lemma}\label{bms2} Допустим, что фактор $V/G$ является многообразием и~что $\Inn{\La}=\La$. Если $v\in V$, $\xi\in\tgt\sm\{0\}$
и~$\ggt_v=\R\xi$, то $\de\cdot\Bm{\bc{\la'\in\La\cln\la'(\xi)\ne0}}\ne\Bm{\bc{\al\in\De\cln\al(\xi)\ne0}}+2$.
\end{lemma}

\begin{proof} В~силу леммы~\ref{dine}, $\de\cdot\dim_{\Cbb}(\xi\wt{V})=\dim(\xi V)\ne\dim[\xi,\ggt]+2=\dim_{\Cbb}[\xi,\ggt_{\Cbb}]+2$, что влечёт
требуемое.
\end{proof}

\begin{imp}\label{nom} Пусть $\Om\subs\La$\т произвольное подмножество. Обозначим через~$H$ подпространство $\ha{\Om}_{\Cbb}\subs\tgt_{\Cbb}^*$. Если
\equ{\begin{array}{c}
\begin{array}{cc}
\dim_{\Cbb}(\tgt_{\Cbb}^*/H)=1;\quad\quad&\Inn{\La}=\La;\\
\de\cdot|\La\sm H|=|\De\sm H|+2;\quad\quad&(\Om-\Om)\cap\De=H^{\perp}\cap\De=\es;
\end{array}\\
(\de=1)\quad\Ra\quad\Br{\br{2\la\notin\De\cup(\De+\De)}\nd\br{(\Om+\Om)\cap\De=\es}},
\end{array}}
то $V/G$\т не многообразие.
\end{imp}

\begin{proof} Вытекает из лемм \ref{sts} и~\ref{bms2}.
\end{proof}

\subsection{Полярные представления}

\begin{df} Представление компактной группы Ли~$G$ в~вещественном пространстве~$V$ называется \textit{полярным}, если существует подпространство
$V'\subs V$, удовлетворяющее соотношениям $GV'=V$ и~$V'\cap\br{T_v(Gv)}=0$ ($v\in V'$).
\end{df}

\begin{lemma}\label{pol} Факторпространство произвольной нетривиальной связной компактной полярной линейной группы гомеоморфно замкнутому
полупространству.
\end{lemma}

\begin{proof} Согласно результатам работы~\cite{CD} (см. теоремы 2.8, 2.9 и~2.10 в~\Ss2), рассматриваемое факторпространство гомеоморфно
факторпространству нетривиальной конечной линейной группы, порождённой (вещественными) отражениями.
\end{proof}

Пусть $G$\т связная односвязная компактная полупростая группа Ли, $\ggt$\т её касательная алгебра, а~$\ta\in\Aut(G)$\т нетривиальный инволютивный
автоморфизм. Далее, рассмотрим присоединённое действие $G^{\ta}\cln\ggt$ и~его ограничение на (очевидно, инвариантное) подпространство
$\ggt^{-d\ta}\subs\ggt$.

\begin{lemma}\label{wey} Фактор $\ggt^{-d\ta}/G^{\ta}$ гомеоморфен замкнутому полупространству.
\end{lemma}

\begin{proof} Согласно теореме~B в~\cite[\Ss1]{Rash}, подгруппа Ли $G^{\ta}\subs G$ связна. Ввиду полупростоты группы Ли~$G$, представление
$G^{\ta}\cln\ggt^{-d\ta}$ нетривиально. Данное представление является также полярным (см.~\cite[\Ss8,~пп.\,8.5---8.6]{T55}). Осталось применить
лемму~\ref{pol}.
\end{proof}

\begin{note} Утверждение леммы~\ref{wey} следует также из результатов работы~\cite{Vin}.
\end{note}

\subsection{Комбинаторные неравенства}

Для натуральных чисел $n,k,m_1\sco m_k$, таких что $m_1\spl m_k=n$, положим
\equ{
\rbmat{n\\m_1\sco m_k}:=\rbmat{n\\m_1}\cdot\rbmat{n-m_1\\m_2}\sd\rbmat{n-m_1\smn m_{k-2}\\m_{k-1}}=\frac{n!}{m_1!\sd m_k!}\in\N.}

\begin{stm}\label{cmb1} Если $n,k,m_1\sco m_k\in\N$ и~$m_1\spl m_k=n$, то неравенство $\rbmat{n\\m_1\sco m_k}<n$ возможно только при $k=1$ и~$m_1=n$.
\end{stm}

\begin{proof} Имеем $\rbmat{n\\m_1}\le\rbmat{n\\m_1\sco m_k}<n$, откуда $m_1=n$.
\end{proof}

\begin{stm}\label{cmb2} Если $n,k,m_1\sco m_k\in\N$, $k>1$ и~$m_1\spl m_k=n$, то неравенство
$\rbmat{n\\m_1\sco m_k}\le2(n-1)$ может выполняться только в~следующих случаях\:
\begin{nums}{-1}
\item $k=2$, $\{m_1,m_2\}=\{1,n-1\}$\~
\item $n=4$, $k=2$, $m_1=m_2=2$.
\end{nums}
\end{stm}

\begin{proof} Допустим, что $k>2$. Тогда $0<m_2<n-m_1$, $0<m_1<n-m_2\le n-1$. Поэтому $\rbmat{n-m_1\\m_2}\ge n-m_1$, а~также
$\rbmat{n-1\\m_1}\ge n-1$. Следовательно,
\equ{
2(n-1)\ge\rbmat{n\\m_1\sco m_k}\ge\rbmat{n\\m_1}\cdot\rbmat{n-m_1\\m_2}\ge\rbmat{n\\m_1}\cdot(n-m_1)=n\cdot\rbmat{n-1\\m_1}\ge n(n-1),}
$(n-2)(n-1)\le0$, $n\le2<k\le m_1\spl m_k=n$. Получили противоречие.

Теперь предположим, что $k=2$ и~$\{m_1,m_2\}\ne\{1,n-1\}$.

Имеем $m_1,m_2\ne1$, $m_1,m_2\ge2$, $\rbmat{n\\m_1}\ge\rbmat{n\\2}$. Кроме того, $n=m_1+m_2\ge4$. Значит,
$2(n-1)\ge\rbmat{n\\m_1,m_2}=\rbmat{n\\m_1}\ge\rbmat{n\\2}=\frac{n(n-1)}{2}\ge\frac{4(n-1)}{2}=2(n-1)$, откуда, во-первых,
$n=4$, а~во-вторых, $\rbmat{n\\m_1}=\rbmat{n\\2}$, $\rbmat{4\\m_1}=\rbmat{4\\2}$, $m_1=m_2=2$.
\end{proof}

\subsection{Орбиты группы Вейля}

Пусть $r>1$\т натуральное число. Положим $n:=r+1\in\N$. Имеем $n\ge3$.

Рассмотрим евклидово пространство~$\R^n$, в~котором стандартный базис $\{\ep_i\}_{i=1}^n$ является ортонормированным, неразложимую систему корней
$\De:=\{\ep_i-\ep_j\cln i\ne j\}\subs\R^n$ типа~$A_r$, её группу Вейля $W\subs\Or(\R^n)$, камеру Вейля $C:=\{x\in\R^n\cln x_1\sge x_n\}\subs\R^n$
и~соответствующую ей систему простых корней $\Pi:=\{\al_1\sco\al_r\}\subs\De$ ($\al_i:=\ep_i-\ep_{i+1}$, $i=1\sco r$).

Решётки $P:=\bc{\la\in\ha{\De}\cln\ha{\la|\al}\in\Z\,\fa\al\in\De}$ и~$Q:=\ha{\De}_{\Z}$ подпространства $\ha{\De}\subs\R^n$ удовлетворяют соотношениям
$WP=P$, $WQ=Q$, $Q\subs P$ и~$|P/Q|=n<\bes$.

Пусть $\ph_1\sco\ph_r\in P$\т фундаментальные веса относительно системы простых корней $\Pi=\{\al_1\sco\al_r\}\subs\De$ (с~учётом нумерации).
Очевидно, что $\{\ph_i\}_{i=1}^r$\т базис решётки $P\subs\ha{\De}$.

Для краткости орбиты действия $W\cln P$ будем называть просто \textit{орбитами}.

Пусть $\La\subs P$\т произвольная орбита. Положим $\Inn{\La}:=\conv(\La)\cap(\La+Q)\subs P$. Ясно, что
\begin{nums}{-1}
\item $|\La\cap C|=1$\~
\item $\La-\La\subs Q$\~
\item $\Inn{\La}\sups\La$\~
\item $W\Inn{\La}=\Inn{\La}$\~
\item для любой орбиты $\La'\subs\Inn{\La}$ имеем $\Inn{\La}{}'\subs\Inn{\La}$\~
\item для любой орбиты $\La'\subs\Inn{\La}$, $\La'\ne\La$, имеем $\Inn{\La}{}'\cap\La=\es$.
\end{nums}

Фиксируем орбиту $\La\subs P\sm\{0\}$.

Имеем $\La\cap C=\{\la\}$, где $\la\in(P\cap C)\sm\{0\}$.

\subsubsection{Основные утверждения}

Пусть $\Pi'\subs\Pi\subs\R^n$\т неразложимая система простых корней порядка $r-2$.

Легко видеть, что $H:=\ba{\{\la\}\cup\Pi'}\subs\ha{\De}\subs\R^n$\т $(r-1)$\д мерное подпространство.

Данный пункт посвящён доказательству нижеследующих лемм \ref{lah} и~\ref{lah1}.

\begin{lemma}\label{lah} Если $\la\notin\De$, $\la\ne\ph_1\sco\ph_r$ и
\equ{
\bgm{\Inn{\La}\sm H}\le4n;\quad(-\La=\La)\lor\Br{\bgm{\Inn{\La}\sm H}\le2n},}
то $\la\in\{2\ph_1,2\ph_r,\ph_1+\ph_{r-1},\ph_2+\ph_r\}$ либо $(r=3)\nd(\la=2\ph_2)$.
\end{lemma}

\begin{lemma}\label{lah1} Предположим, что $r\ge8$, $\la=\ph_j$, где $j\in\N$ и~$3\le j\le r-2$, а~также $\Pi'=\{\al_1\sco\al_{r-2}\}\subs\Pi$. Тогда
$\bgm{\Inn{\La}\sm H}>4n$.
\end{lemma}

Имеем $\la\in\br{\ha{\De}\cap C}\sm\{0\}$, $\la_1>0>\la_n$, откуда
\equ{
\ha{\Pi'}=\bc{x\in\ha{\De}\cln x_p=x_q=0},\quad H=\bc{x\in\ha{\De}\cln x_q=cx_p},}
где $(p,q)\in\{1\sco n\}^2$\т одна из пар $(1,2)$, $(1,n)$ и~$(n,n-1)$ (и, таким образом, $\la_p\ne0$), а~$c$\т число $\frac{\la_q}{\la_p}\in\R$.

Мы начнём с~доказательства леммы~\ref{lah}, предварительно доказав несколько вспомогательных утверждений.

Пусть $x\in P\cap C$\т некоторый вектор.

Положим $K:=\{x_1\sco x_n\}\subs\R$, $k:=|K|\in\N$, $m(t):=\Bm{\bc{i\in\{1\sco n\}\cln x_i=t}}\in\N$ ($t\in K$)
и~$m_0:=\max\bc{m(t)}_{t\in K}\in\N$.

\begin{prop}\label{est} Допустим, что $k\ge3$. Тогда
\begin{nums}{-1}
\item $\bgm{(Wx)\sm H}\ge2n-2$\~
\item если $m_0\ne n-2$, то $\bgm{(Wx)\sm H}\ge5(n-2)$\~
\item если $k\ge4$, то $\bgm{(Wx)\sm H}\ge4n$.
\end{nums}
\end{prop}

\begin{proof} Положим $L(t'):=\BC{t\"\in K\sm\{ct'\}\cln\br{m(t')=1}\Ra(t\"\ne t')}\subs K$ ($t'\in K$), а~также
$L:=\bc{(t',t\")\in K^2\cln t\"\in L(t')}\subs K^2$ и~$K_0:=\bc{t\in K\cln ct\in K\sm\{t\},m(t)=1}\subs K$. Ясно, что
\eqn{\label{dt}
\begin{split}
\fa\tau\in L\quad\quad\quad&d(\tau):=\Bm{\bc{y\in Wx\cln(y_p,y_q)=\tau}}\ge1;\\
&\bgm{(Wx)\sm H}=\sums{\tau\in L}d(\tau);\\
\fa t\in K\quad\quad\quad&\bgm{L(t)}\ge k-2;\\
\fa t\in K\sm K_0\quad\quad\quad&\bgm{L(t)}>k-2.
\end{split}}
Отсюда $|L|\ge k(k-2)+\br{k-|K_0|}\ge k(k-2)$, $|K_0|\ge k(k-2)+k-|L|=k(k-1)-|L|$.

Покажем, что
\eqn{\label{L4}
\begin{array}{c}
|L|\ge4;\\
(|L|=4)\quad\Ra\quad(m_0=n-2).
\end{array}}

Допустим, что $|L|\le4$. Тогда $k(k-2)\le|L|\le4<4\cdot(4-2)$, $k=3$, $|K_0|\ge6-|L|\ge2$, $\Bm{\bc{t\in K\cln m(t)=1}}\ge|K_0|\ge2$. Кроме того,
$n\ge|K|=k=3$. Значит, $m_0=n-2$. Если при этом $|L|\le3$, то $|K_0|\ge6-|L|\ge3=|K|$, $K_0=K$,
\begin{gather*}
\fa t\in K\\
ct\in K\sm\{t\},\quad ct\ne t,\quad t\ne0,\quad c\ne1,\quad c^3\ne1,\quad c^3t\ne t;\\
K\subs\R\sm\{0\};\\
\fa t\in K\quad\quad\quad ct\in K\sm\{t\}\subs\R\sm\{0\},\quad c\ne0,
\end{gather*}
и~тем самым возникает биективное отображение $K\to K,\,t\to ct$, третья степень которого не имеет неподвижных точек, в~то время как $|K|=3$. Получили
противоречие.

Тем самым нами установлены соотношения~\eqref{L4}.

Покажем, что
\eqn{\label{ln}
\begin{array}{c}
l:=\Bm{\bc{\tau\in L\cln d(\tau)<n-2}}\le2;\\
(l>0)\quad\Ra\quad(m_0=n-2).
\end{array}}

Фиксируем число $t_0\in K$, для которого $m(t_0)=m_0$ (такое существует). Кроме того, положим
$L_0:=\bc{(t',t\")\in K^2\cln t',t\"\ne t_0,t'\ne t\"}\subs K^2$.

Пусть $(t',t\")\in L$\т произвольная пара, такая что $d(t',t\")<n-2$. В~силу утверждения~\ref{cmb1} и~неравенства $k\ge3$, имеем $k=3$,
$K=\{t,t',t\"\}\subs\R$, $t'\ne t$, $t\"\ne t$, $t'\ne t\"$, $m(t)=n-2>d(t',t\")\ge1=m(t')=m(t\")$, $m_0=n-2$, $t_0=t\in K\sm\{t',t\"\}$,
$(t',t\")\in L_0$, причём $|L_0|=(k-1)(k-2)=2$.

Тем самым мы установили соотношения~\eqref{ln}.

Ясно, что $n\ge|K|=k\ge3$. Теперь, пользуясь \eqref{dt}---\eqref{ln}, получаем, что
\begin{gather*}
\bgm{(Wx)\sm H}\ge(n-2)\cdot\br{|L|-l}+l=(n-2)\cdot|L|-(n-3)l=\\
=4(n-2)-2(n-3)+(n-2)\cdot\br{|L|-4}+(n-3)(2-l)\ge\\
\ge4(n-2)-2(n-3)=2n-2,
\end{gather*}
причём в~случае $m_0\ne n-2$ имеем $|L|\ge5$ и~$l=0$, откуда
\equ{
\bgm{(Wx)\sm H}\ge(n-2)\cdot|L|-(n-3)l=(n-2)\cdot|L|\ge5(n-2).}
Допустим, что $k\ge4$. Тогда $n\ge k\ge4$, $|L|\ge k(k-2)\ge8$, а~также $m_0\le n-(k-1)<n-2$, и~поэтому
$\bgm{(Wx)\sm H}\ge(n-2)\cdot|L|\ge8(n-2)=4n+4(n-4)\ge4n$.

Это завершает доказательство.
\end{proof}

\begin{imp}\label{deh} Справедливо неравенство $|\De\sm H|\ge2n-2$.
\end{imp}

\begin{proof} Имеем $x:=\ep_1-\ep_n\in\De\cap C$, а~также $\bgm{\{x_1\sco x_n\}}=3$. Осталось воспользоваться предложением~\ref{est}.
\end{proof}

\begin{prop}\label{la2} Для всякого вектора $x\in(P\cap C)\sm\{0\}$ выполнено неравенство $\bgm{(Wx)\sm H}\ge2$, которое может обращаться в~равенство
лишь в~следующих случаях\:
\begin{nums}{-1}
\item\label{fst} $x_2\seq x_n$\~
\item\label{lst} $x_1\seq x_{n-1}$\~
\item\label{tt} $n=4$, $x_1=x_2$, $x_3=x_4$.
\end{nums}
\end{prop}

\begin{proof} Поскольку тавтологическое представление $W\cln\ha{\De}$ неприводимо, справедливы соотношения
\begin{align}
\sums{y\in Wx}y=0;\label{sum0}\\
\ha{Wx}=\ha{\De};\label{gen}\\
\caps{w\in W}(wH)=0.\label{inth}
\end{align}
Согласно~\eqref{sum0}, если $\bgm{(Wx)\sm H}\le1$, то $Wx\subs H$, $\ha{Wx}\subs H$, что противоречит~\eqref{gen}.

Тем самым доказано, что $\bgm{(Wx)\sm H}\ge2$.

Допустим, что $\bgm{(Wx)\sm H}=2$.

В~силу~\eqref{inth}, найдутся элементы $w_1\sco w_r\in W$, для которых $\capl{i=1}{r}(w_iH)=0$. Имеем $0\notin Wx$, $Wx=\cupl{i=1}{r}\br{(Wx)\sm(w_iH)}$,
$|Wx|\le\suml{i=1}{r}\bgm{(Wx)\sm(w_iH)}=\suml{i=1}{r}\bgm{(Wx)\sm H}=2(n-1)$. Осталось применить утверждение~\ref{cmb2}.
\end{proof}

\begin{prop}\label{elst} Пусть $x\in(P\cap C)\sm\{0\}$\т некоторый вектор, а~$\La'$\т орбита $Wx\subs P$. Тогда $\bgm{\Inn{\La}{}'\sm H}\ge2$, причём
точное равенство $\bgm{\Inn{\La}{}'\sm H}=2$ возможно только в~следующих случаях\:
\begin{nums}{-1}
\item $x_1-x_2=1$, $x_2\seq x_n$\~
\item $x_1\seq x_{n-1}$, $x_{n-1}-x_n=1$\~
\item $n=4$, $x_1=x_2$, $x_2-x_3=1$, $x_3=x_4$.
\end{nums}
\end{prop}

\begin{proof} Из предложения~\ref{la2} и~соотношения $\Inn{\La}{}'\sups\La'$ вытекает неравенство $\bgm{\Inn{\La}{}'\sm H}\ge2$, которое может обращаться
в~равенство только в~случаях \ref{fst}---\ref{tt}, перечисленных в~формулировке предложения~\ref{la2}, при отсутствии орбит
в~($W$\д инвариантном) подмножестве $\Inn{\La}{}'\bbl\br{\La'\sqcup\{0\}}\subs P$, т.\,е. при $\Inn{\La}{}'\sm\{0\}=\La'$.

Допустим, что $\bgm{\Inn{\La}{}'\sm H}=2$. Тогда, во-первых, справедливо равенство $\Inn{\La}{}'\sm\{0\}=\La'$, а~во-вторых, выполнено одно из условий
\ref{fst}---\ref{tt} формулировки предложения~\ref{la2}, откуда $x\notin\De$. Далее, для всякого $\al\in\De$, такого что $\ha{x|\al}>0$, имеем
$x-\al\in\Inn{\La}{}'\sm\{0\}=\La'$, что влечёт равенство $\ha{x|\al}=1$. В~частности, $\ba{x|(\ep_1-\ep_n)}=1$, $x_1-x_n=1$.

Теперь предложение полностью доказано.
\end{proof}

\begin{prop}\label{esn} Пусть $x\in P\cap C$\т некоторый вектор, а~$\La'$\т орбита $Wx\subs P$. Если $x\notin\De$ и~$\bgm{\{x_1\sco x_n\}}\ge3$, то
имеет место неравенство $\bgm{\Inn{\La}{}'\sm H}\ge2n$, которое может обращаться в~равенство лишь в~следующих случаях\:
\begin{nums}{-1}
\item $x_1=x_2$, $x_3\seq x_{n-1}$, $x_2-x_3=x_{n-1}-x_n=1$\~
\item $x_2\seq x_{n-2}$, $x_{n-1}=x_n$, $x_1-x_2=x_{n-2}-x_{n-1}=1$.
\end{nums}
\end{prop}

\begin{proof} Поскольку $\bgm{\{x_1\sco x_n\}}\ge3$, имеем $x_1-x_n>1$. Это позволяет выбрать среди всех пар $(i_1,i_2)\in\{1\sco n\}^2$,
удовлетворяющих неравенству $x_{i_1}-x_{i_2}>1$, пару $(i_1,i_2)$ с~минимальным возможным значением $i_2-i_1$.

Очевидно, что $i_2>i_1$. Кроме того, если $i=i_1+1\sco i_2-1$\т произвольное число, то $x_{i_1}-x_i\le1$ и~$x_i-x_{i_2}\le1$, в~то время как
$(x_{i_1}-x_i)+(x_i-x_{i_2})=x_{i_1}-x_{i_2}>1$, откуда $x_{i_1}-x_i=x_i-x_{i_2}=1$. Таким образом,
\eqn{\label{inc}
\begin{array}{c}
x_{i_1}-1\ge x_{i_2}+1;\\
(i_2-i_1>1)\quad\Ra\quad(x_{i_1}-1=x_{i_1+1}\seq x_{i_2-1}=x_{i_2}+1).
\end{array}}

Положим $\al:=\ep_{i_1}-\ep_{i_2}\in\De$. Из соотношений $\ha{x|\al}=x_{i_1}-x_{i_2}>1$ и~$x\notin\De$ следует, что
$y:=x-\al\in\Inn{\La}{}'\bbl\br{\La'\sqcup\{0\}}$. Обозначим через~$\La\"$ орбиту $Wy\subs P$.

Имеем $x\in C$, $x_1\sge x_n$, и, согласно~\eqref{inc},
\begin{gather}
x_1\sge x_{i_1-1}>x_{i_1}-1\ge x_{i_1+1}\sge x_{i_2-1}\ge x_{i_2}+1>x_{i_2+1}\sge x_n,\notag\\
(i_2-i_1>1)\quad\Ra\quad y_{i_1}\seq y_{i_2},\label{eqy}
\end{gather}
$y_1\sge y_{i_1-1}>y_{i_1}\sge y_{i_2}>y_{i_2+1}\sge y_n$. Следовательно, $y\in C$. Кроме того, выполнено включение
$\Inn{\La}{}'\sups\La'\sqcup\Inn{\La}{}\"$, а~значит, и~неравенство $\bgm{\Inn{\La}{}'\sm H}\ge|\La'\sm H|+\bgm{\Inn{\La}{}\"\sm H}$.

Допустим, что $\bgm{\Inn{\La}{}'\sm H}\le2n$.

Из предложений \ref{est} и~\ref{elst} вытекает, что $\bgm{\Inn{\La}{}'\sm H}=2n$, $y_1\seq y_j$, $y_{j+1}\seq y_n$, $y_j-y_{j+1}=1$,
$j\in\{1\sco n-1\}$, причём $j\in\{1,n-1\}$ либо $n=4$ и~$j=2$.

Заметим, что
\begin{nums}{-1}
\item как уже отмечалось, $y\ne0$, и, ввиду~\eqref{eqy}, $(i_1,i_2)\ne(1,n)$\~
\item если $i_1>1$, то $y_{i_1-1}>y_{i_1}$, $j=i_1-1<i_2-1\le n-1$, $(j=1)\lor\br{(n=4)\nd(j=2)}$\~
\item если $i_2<n$, то $y_{i_2}>y_{i_2+1}$, $j=i_2>i_1\ge1$, $(j=n-1)\lor\br{(n=4)\nd(j=2)}$.
\end{nums}
Поскольку $i_1-1<i_1<i_2$, выполняется ровно одно из равенств $i_1=1$ и~$i_2=n$.

Возможны следующие случаи.

\begin{cass}{ri} $i_1>1$, $i_2=n$, $j=i_1-1$ и~$j=1$.
\end{cass}

Имеем $i_1=2$, $y_1-1=y_2\seq y_n$, $x_1-1=x_2-1=x_3\seq x_{n-1}=x_n+1$.

\begin{cass}{ri2} $i_1>1$, $i_2=n$, $j=i_1-1$, $n=4$ и~$j=2$.
\end{cass}

Имеем $i_1=3$, $y_1=y_2=y_3+1=y_4+1$, $x_1=x_2=x_3=x_4+2$.

\begin{cass}{le} $i_1=1$, $i_2<n$, $j=i_2$ и~$j=n-1$.
\end{cass}

Имеем $i_2=n-1$, $y_1\seq y_{n-1}=y_n+1$, $x_1-1=x_2\seq x_{n-2}=x_{n-1}+1=x_n+1$.

\begin{cass}{le2} $i_1=1$, $i_2<n$, $j=i_2$, $n=4$ и~$j=2$.
\end{cass}

Имеем $i_2=2$, $y_1-1=y_2-1=y_3=y_4$, $x_1-2=x_2=x_3=x_4$.

Осталось воспользоваться тем, что $\bgm{\{x_1\sco x_n\}}\ge3$.
\end{proof}

\begin{prop}\label{es2n} Пусть $x\in P\cap C$\т некоторый вектор, а~$\La'$\т орбита $Wx\subs P$. Если
\eqn{\label{el2}
\bgm{\{x_1\sco x_n\}}=2,}
$x_1-x_n>1$ и~$\bgm{\Inn{\La}{}'\sm H}\le2n$, то выполняется одно из следующих условий\:
\begin{nums}{-1}
\item $x_1-x_2=2$, $x_2\seq x_n$\~
\item $x_1\seq x_{n-1}$, $x_{n-1}-x_n=2$\~
\item $n=4$, $x_1=x_2=1$, $x_3=x_4=-1$.
\end{nums}
\end{prop}

\begin{proof} По условию $x_1\seq x_j$, $x_{j+1}\seq x_n$ и~$x_j-x_{j+1}\ge2$, где $j\in\{1\sco n-1\}$.

Положим $\al:=\al_j\in\De$. Имеем $\ha{x|\al}=x_j-x_{j+1}>1$, и~поэтому $y:=x-\al\in\Inn{\La}{}'\sm\La'$. Кроме того, $x\in C$, $x_1\sge x_n$,
$x_1\sge x_{j-1}>x_j-1\ge x_{j+1}+1>x_{j+2}\sge x_n$, $y_1\sge y_{j-1}>y_j\ge y_{j+1}>y_{j+2}\sge y_n$, $y\in C$.

Обозначим через~$\La\"$ орбиту $Wy\subs P$. Легко видеть, что $\Inn{\La}{}'\sups\La'\sqcup\Inn{\La}{}\"$. Следовательно,
$2n\ge\bgm{\Inn{\La}{}'\sm H}\ge|\La'\sm H|+\bgm{\Inn{\La}{}\"\sm H}$. Согласно предложению~\ref{la2}, $|\La'\sm H|\ge2$, что влечёт
неравенство $\bgm{\Inn{\La}{}\"\sm H}\le2n-2<2n$. Пользуясь предложением~\ref{esn}, получаем, что либо $\bgm{\{y_1\sco y_n\}}\le2$, либо $y\in\De$.

Допустим, что $\bgm{\{y_1\sco y_n\}}\le2$.

Для целых чисел $k_1:=\bgm{\{y_1\sco y_{j-1}\}}$, $k_2:=\bgm{\{y_j,y_{j+1}\}}$, $k_3:=\bgm{\{y_{j+2}\sco y_n\}}$ имеем
$k_1,k_3\ge0$, $k_2\ge1$, $(k_1=0)\Ra(j=1)$ и~$(k_3=0)\Ra(j=n-1)$. Далее, $n-1>1$, откуда
\begin{gather*}
k_1+k_3\ge1;\\
(k_1+k_3=1)\quad\Ra\quad\br{j\in\{1,n-1\}}.
\end{gather*}
При этом $(k_1+k_3)+k_2=\bgm{\{y_1\sco y_n\}}\le2$ и~$k_2\ge1$. Значит, $k_1+k_3=1$, $j\in\{1,n-1\}$, а~также $k_2=1$, $y_j=y_{j+1}$, $x_j-x_{j+1}=2$.

Теперь предположим, что $y\in\De$.

В~силу~\eqref{el2}, $0\notin\{x_1\sco x_n\}$. Из этого, а~также из соотношений $x=y+\al\in\De+\De$ и~$x\in C$ вытекает, что
\begin{gather*}
n\le4;\\
(n=4)\quad\Ra\quad(x_1=x_2=1,\ x_3=x_4=-1).
\end{gather*}
Допустим, что $n=3$. Тогда $r=2$,
\eqn{\label{line}
\dim H=r-1=1.}
Далее, $\La\"=\De$, $\Inn{\La}{}'\sups\La'\sqcup\De$,
$|\La'|+|\De|\le|\La'\cap H|+|\De\cap H|+\bgm{\Inn{\La}{}'\sm H}\le|\La'\cap H|+|\De\cap H|+2n$, $3+n(n-1)\le|\La'\cap H|+|\De\cap H|+2n$,
$|\La'\cap H|+|\De\cap H|\ge n(n-3)+3=3$, и, согласно~\eqref{line}, $|\La'\cap H|,|\De\cap H|\le2<3\le|\La'\cap H|+|\De\cap H|$,
$|\La'\cap H|,|\De\cap H|>0$, $\La'\cap H,\De\cap H\ne\es$. В~то же время, ввиду~\eqref{el2}, ни один корень системы корней $\De\subs\R^n$
не пропорционален ни одному из векторов орбиты $\La'=Wx\subs P$, что противоречит~\eqref{line}.

Таким образом, предложение доказано.
\end{proof}

\begin{imp}\label{ampl} Пусть $x\in P\cap C$\т некоторый вектор. Обозначим через~$\La'$ орбиту $Wx\subs P$. Если $x_1-x_n>2$, то
$\bgm{\Inn{\La}{}'\sm H}>2n$ и~$|\La'\sm H|\ge2$.
\end{imp}

\begin{proof} Вытекает из предложений \ref{la2}, \ref{esn} и~\ref{es2n}.
\end{proof}

\begin{prop}\label{eson} Пусть $x\in P\cap C$\т некоторый вектор, а~$\La'$\т орбита $Wx\subs P$. Если $x\notin\De$, $x_1-x_n>1$, $-\La'=\La'$
и~$\bgm{\Inn{\La}{}'\sm H}\le4n$, то $n=4$, $x_1=x_2=1$ и~$x_3=x_4=-1$.
\end{prop}

\begin{proof} Для всякого $i=1\sco n$ справедливо равенство $x_{n+1-i}=-x_i$. Далее, поскольку $n-1\ge2$, имеем $x_2\ge x_{n-1}=-x_2$, $x_2\ge0$.

Возможны следующие случаи.

\begin{cass}{big} $x_1-x_n>2$ и~$x_2>0$.
\end{cass}

\begin{cass}{mid} $x_1-x_n>2$ и~$x_2=0$.
\end{cass}

\begin{cass}{sma} $x_1-x_n=2$.
\end{cass}

Вначале разберём случай~\ref{big}.

Имеем $x_1=\frac{1}{2}(x_1-x_n)>1$.

Положим $\al:=\ep_1-\ep_n\in\De$. Пользуясь соотношениями $\ha{x|\al}=x_1-x_n>2$ и~$x\notin\De$, получаем, что
$y:=x-\al\in\Inn{\La}{}'\bbl\br{\La'\sqcup\{0\}}$. Обозначим через~$\La\"$ орбиту $Wy\subs P$. Имеем $\La\"\cap C=\bc{\wt{y}}$,
$\wt{y}\in(P\cap C)\sm\{0\}$. Кроме того, $\Inn{\La}{}'\sups\La'\sqcup\La\"$,
$4n\ge\bgm{\Inn{\La}{}'\sm H}\ge|\La'\sm H|+|\La\"\sm H|$, и, в~силу предложения~\ref{la2}, $|\La'\sm H|,|\La\"\sm H|<4n$.

Согласно предложению~\ref{est}, $\bgm{\{x_1\sco x_n\}},\Bm{\bc{\wt{y}_1\sco\wt{y}_n}}<4$, $\bgm{\{y_1\sco y_n\}}<4$. При этом
$\{x_1\sco x_n\}\sups\{x_1,x_2,x_{n-1},x_n\}=\{\pm x_1,\pm x_2\}$, $\{y_1\sco y_n\}=\{y_1,x_2\sco x_{n-1},y_n\}$, $y_1=x_1-1>0$, $y_n=x_n+1=-x_1+1=-y_1$,
$\{y_1\sco y_n\}\sups\{\pm y_1,\pm x_2\}$. Таким образом, $\bgm{\{\pm x_1,\pm x_2\}}<4$ и~$\bgm{\{\pm y_1,\pm x_2\}}<4$.
В~то же время $x_1,y_1,x_2>0$. Отсюда следует, что $x_2=x_1$ и~$x_2=y_1=x_1-1$.

Тем самым мы в~случае~\ref{big} пришли к~противоречию.

Теперь разберём случай~\ref{mid}.

Имеем $x_{n-1}=-x_2=0$, $x_2\seq x_{n-1}=0$, $x=x_1(\ep_1-\ep_n)$. Далее, $x_1=x_1-x_2\in\Z$, а~также $x_1=\frac{1}{2}(x_1-x_n)>1$, откуда $x_1\ge2$.
Заметим, что $x_1-1\ge1>0>-x_1$, и, как следствие, $y:=(x_1-1)\ep_1+\ep_2-x_1\ep_n\in\br{\Inn{\La}{}'\sm\La'}\cap C\subs P$.

Орбиты $\La\":=Wy\subs P$ и~$-\La\"\subs P$ различны, поскольку $y_1+y_n=(x_1-1)-x_1\ne0$. Отсюда
$\Inn{\La}{}'\sups\La'\sqcup(-\La\")\sqcup\Inn{\La}{}\"$, $4n\ge\bgm{\Inn{\La}{}'\sm H}\ge|\La'\sm H|+|\La\"\sm H|+\bgm{\Inn{\La}{}\"\sm H}$.
При этом
\begin{nums}{-1}
\item $\bgm{\{x_1\sco x_n\}}=\bgm{\{\pm x_1,0\}}=3$, и, согласно предложению~\ref{est}, $|\La'\sm H|\ge2n-2$\~
\item $y_1-y_n=(x_1-1)+x_1=2x_1-1\ge3$, и, в~силу следствия~\ref{ampl}, $|\La\"\sm H|+\bgm{\Inn{\La}{}\"\sm H}>2n+2$.
\end{nums}
Значит, $|\La'\sm H|+|\La\"\sm H|+\bgm{\Inn{\La}{}\"\sm H}>(2n-2)+(2n+2)=4n$. Получили противоречие.

Наконец, разберём случай~\ref{sma}.

Имеем $x_1=\frac{1}{2}(x_1-x_n)=1$, $x=(\ep_1\spl\ep_j)-(\ep_{n+1-j}\spl\ep_n)$, $j\in\N$, $j\le\frac{n}{2}$. По условию $x\notin\De$, откуда $j\ge2$,
$n\ge4$. Если $n=4$, то $j=2$, $x_1=x_2=1$, $x_3=x_4=-1$.

Допустим, что $n>4$. Положим $y:=(\ep_1+\ep_2)-(\ep_{n-1}+\ep_n)\in P\cap C$. Легко видеть, что $\Inn{\La}{}'\sups(Wy)\sqcup\De$. Значит,
$4n\ge\bgm{\Inn{\La}{}'\sm H}\ge\bgm{(Wy)\sm H}+|\De\sm H|$, и, согласно следствию~\ref{deh}, $\bgm{(Wy)\sm H}\le4n-|\De\sm H|\le4n-(2n-2)=2n+2$.
Далее, $n-4>0$, $\bgm{\{y_1\sco y_n\}}=3$, и, кроме того, $n-2>2$, $\max\{2,n-4\}<n-2$. Из этого, а~также из предложения~\ref{est} вытекает, что
$\bgm{(Wy)\sm H}\ge5(n-2)=(2n+2)+3(n-4)>2n+2$. Получили противоречие.

Теперь предложение полностью доказано.
\end{proof}

Применяя к~вектору $\la\in(P\cap C)\sm\{0\}$ и~его орбите $\La\subs P\sm\{0\}$ предложения \ref{esn}---\ref{eson}, получаем утверждение
леммы~\ref{lah}.

Докажем теперь лемму~\ref{lah1}.

Согласно условию, $j\in\{3\sco n-3\}$, $\la_1\seq\la_j$, $\la_{j+1}\seq\la_n$ и~$\la_j-\la_{j+1}=1$. В~частности, $\la_{n-1}=\la_n$, и, значит,
$H=\bc{x\in\ha{\De}\cln x_{n-1}=x_n}$, $|\La\sm H|=\rbmat{2\\1}\cdot\rbmat{n-2\\j-1}$. При этом $2\le j-1\le(n-2)-2$, откуда
$\rbmat{n-2\\j-1}\ge\rbmat{n-2\\2}=\frac{(n-2)(n-3)}{2}$. Следовательно, $|\La\sm H|\ge(n-2)(n-3)>n(n-5)=n(r-4)=4n+n(r-8)\ge4n$.

Тем самым лемма~\ref{lah1} доказана.

\subsubsection{Дополнительные утверждения}

В~этом пункте будут приведены (с~доказательствами) предложения \ref{alg}---\ref{ths}.

Для произвольных попарно различных чисел $i_1\sco i_k\in\{1\sco n\}$, где $k=1\sco n$, положим
$\la_{(i_1\sco i_k)}:=(\ep_{i_1}\spl\ep_{i_k})-\frac{k}{n}\cdot(\ep_1\spl\ep_n)\in P$. Далее, для произвольных попарно различных чисел
$i,i_1,i_2\in\{1\sco n\}$ положим
\begin{gather*}
\la_{(i_1,i_2)}^{(i)}:=\la_{(i_1)}+(\ep_{i_2}-\ep_i)=\la_{(i_2)}+(\ep_{i_1}-\ep_i)=\\=(\ep_{i_1}+\ep_{i_2}-\ep_i)-\frac{1}{n}\cdot(\ep_1\spl\ep_n)\in P.
\end{gather*}

\begin{prop}\label{alg} Если $r>2$ и~$\la=\ph_2+\ph_r$, то найдутся подмножество $\Om\subs\La$ и~гиперплоскость $H\subs\ha{\De}$, такие что $\ha{\Om}=H$,
$(\Om-\Om)\cap\De=\es$ и~$2\cdot\bgm{\Inn{\La}\sm H}>|\De\sm H|+6$.
\end{prop}

\begin{proof} Имеем $n=r+1\ge4$, $\la=\la_{(1,2)}^{(n)}$ и~$\Inn{\La}=\La\sqcup\{\la_{(1)}\sco\la_{(n)}\}$.

Положим $\Om:=\bc{\la_{(i-1,i)}^{(i+1)}\cln i=2\sco r}\subs\La$ и~$H:=\ha{\Om}\subs\ha{\De}$.

Докажем, что подмножество $\Om\subs\La$ и~подпространство $H\subs\ha{\De}$ искомые.

Произвольный вектор вида $(\ep_{j-1}+\ep_j-\ep_{j+1})-(\ep_{i-1}+\ep_i-\ep_{i+1})\in P$ ($2\le i<j\le r$) имеет $(i+1)$\д ю координату, равную~$2$, при
$j\le i+2$ и~шесть ненулевых координат при $j>i+2$. Значит, $(\Om-\Om)\cap\De=\es$.

Подпространство $H':=H^{\perp}\cap\ha{\De}\subs\R^n$ задаётся уравнениями $x_{i-1}+x_i=x_{i+1}$, где $i=2\sco r$, и~$x_1\spl x_n=0$. При этом
$\dim H\le|\Om|=r-1$, $\dim H'=r-\dim H\ge1$.

Пусть $x\in H'$\т произвольный ненулевой вектор.

Покажем, что $x_1,x_2,x_3,x_4\ne0$.

Имеем $x_{i-1}+x_i=x_{i+1}$ ($i=2\sco r$) и~$x_1\spl x_n=0$, откуда
\begin{gather}
\fa i=2\sco r\quad\quad\quad x_{i-1}x_i\le x_{i-1}x_i+x_i^2=x_ix_{i+1};\label{seg}\\
2x_3+\Br{\suml{i=4}{n}x_i}=(x_1+x_2+x_3)+\Br{\suml{i=4}{n}x_i}=\suml{i=1}{n}x_i=0;\notag\\
4x_3^2+4\cdot\Br{\suml{i=4}{n}x_3x_i}+\Br{\suml{i_1,i_2=4}{n}x_{i_1}x_{i_2}}=0.\label{sqs}
\end{gather}

Предположим, что найдётся пара $(i_1,i_2)\in\{3\sco n\}^2$, удовлетворяющая неравенствам $i_1\le i_2$ и~$x_{i_1}x_{i_2}<0$. Среди всех указанных пар
выберем пару $(i_1,i_2)$ с~минимально возможным значением $i_1+i_2$. Имеем $i_1<i_2$. Далее, если $i_2\ge5$, то $3\le i_2-2<i_2-1<n$
и~$i_1+(i_2-2)<i_1+(i_2-1)<i_1+i_2$, откуда $x_{i_1}x_{i_2-2},x_{i_1}x_{i_2-1}\ge0$, $x_{i_1}(x_{i_2-2}+x_{i_2-1})\ge0$, $x_{i_1}x_{i_2}\ge0$, что
противоречит предположению. Значит, $i_2<5$, $3\le i_1<i_2<5$, и~поэтому $(i_1,i_2)=(3,4)$, $x_3x_4<0$. В~силу~\eqref{seg},
$x_1x_2\le x_2x_3\le x_3x_4<0$, $x_1,x_2,x_3,x_4\ne0$.

Теперь допустим, что для любых $i_1,i_2\in\{3\sco n\}$ ($i_1\le i_2$) выполнено неравенство $x_{i_1}x_{i_2}\ge0$. Тогда, согласно~\eqref{sqs},
$x_3^2\seq x_n^2=0$, $x_3\seq x_n=0$, $x_2=x_4-x_3=0$, $x_1=x_3-x_2=0$, $x_1\seq x_n=0$, в~то время как $x\ne0$. Получили противоречие.

Тем самым нами установлено, что $\dim H'\ge1$, причём для всякого ненулевого вектора $x\in H'$ справедливо соотношение $x_1,x_2,x_3,x_4\ne0$. Как
следствие, $H'=\R x\subs\ha{\De}$ ($x\in\ha{\De}$, $x_1,x_2,x_3,x_4\ne0$), $\dim H'=1$,
$H=(H')^{\perp}\cap\ha{\De}=\bc{y\in\ha{\De}\cln(y,x)=0}\subs\R^n$, $\dim H=r-1$.

Осталось доказать, что $2\cdot\bgm{\Inn{\La}\sm H}>|\De\sm H|+6$.

Для любого $i=1,2,3,4$ имеем $(\la_{(i)},x)=(\ep_i,x)=x_i\ne0$, откуда $\la_{(i)}\in\Inn{\La}\sm H$. Значит,
$\bgm{\Inn{\La}\sm H}\ge|\La\sm H|+4>|\La\sm H|+3$.

Пусть $I=\{i_1,i_2\}\subs\{1\sco n\}$\т произвольное двухэлементное подмножество.

Докажем, что найдётся число $i\in\{1\sco n\}\sm I$, для которого $x_i\ne x_{i_1}+x_{i_2}$.

Предположим, что для любого $i\in\{1\sco n\}\sm I$ имеет место равенство $x_i=x_{i_1}+x_{i_2}$. Тогда
$0=\suml{i=1}{n}x_i=(x_{i_1}+x_{i_2})+(n-2)(x_{i_1}+x_{i_2})=(n-1)(x_{i_1}+x_{i_2})$, $x_{i_1}+x_{i_2}=0$. Таким образом, все числа $x_i\in\R$, где
$i\in\{1\sco n\}$ и~$i\ne i_1,i_2$, нулевые, что невозможно ввиду соотношения $x_1,x_2,x_3,x_4\ne0$.

Согласно вышесказанному, для любых различных чисел $i_1,i_2\in\{1\sco n\}$ найдётся число $i\in\{1\sco n\}\bbl\{i_1,i_2\}$, такое что
$x_i\ne x_{i_1}+x_{i_2}$ (и~поэтому $(\ep_{i_1}+\ep_{i_2}-\ep_i,x)\ne0$, $(\la_{(i_1,i_2)}^{(i)},x)\ne0$, $\la_{(i_1,i_2)}^{(i)}\in\La\sm H$). Значит,
$|\La\sm H|\ge C_n^2$, $\bgm{\Inn{\La}\sm H}>|\La\sm H|+3\ge C_n^2+3$ и, следовательно,
$2\cdot\bgm{\Inn{\La}\sm H}>2\cdot C_n^2+6=|\De|+6\ge|\De\sm H|+6$.
\end{proof}

\begin{prop}\label{all} Если $r=3$ и~$\la=\ph_1+\ph_2+\ph_3$, то
\begin{nums}{-1}
\item $2\la\notin\De\cup(\De+\De)$\~
\item найдутся подмножество $\Om\subs\La$ и~гиперплоскость $H\subs\ha{\De}$, такие что $\ha{\Om}=H$, $(\Om+\Om)\cap\De=(\Om-\Om)\cap\De=\es$
и~$\bgm{\Inn{\La}\sm H}>|\De\sm H|+6$.
\end{nums}
\end{prop}

\begin{proof} Имеем $\la=\frac{1}{2}\cdot\br{3(\ep_1-\ep_4)+(\ep_2-\ep_3)}$, $2\la=3(\ep_1-\ep_4)+(\ep_2-\ep_3)\notin\De\cup(\De+\De)$,
$|\La|=24$ и~$|\De|=12$.

Положим $\Om:=\BC{\frac{1}{2}\cdot\br{3(\ep_1-\ep_4)+(\ep_2-\ep_3)},\frac{1}{2}\cdot\br{3(\ep_1-\ep_2)+(\ep_4-\ep_3)}}\subs\La$. Как легко заметить,
$(\Om+\Om)\cap\De=(\Om-\Om)\cap\De=\es$ и~$H:=\ha{\Om}=\bc{x\in\ha{\De}\cln x_1+3x_3=0}\subs\ha{\De}$. Далее, $\La\cap H=\Om\sqcup(-\Om)\subs\La$,
$|\La\cap H|=4$, $|\La\sm H|=20>|\De|+6\ge|\De\sm H|+6$.
\end{proof}

\begin{prop}\label{thi} Если $r=6$ и~$\la=\ph_3$, то существуют подмножество $\Om\subs\La$ и~гиперплоскость $H\subs\ha{\De}$, такие что $\ha{\Om}=H$,
$(\Om-\Om)\cap\De=\es$ и~$2\cdot\bgm{\Inn{\La}\sm H}>|\De\sm H|+6$.
\end{prop}

\begin{proof} Согласно условию, $\la=\la_{(1,2,3)}$ и~$|\La|=35$.

Положим $\Om:=\bc{\la_{(1,2,5)},\la_{(3,4,5)},\la_{(1,4,6)},\la_{(2,3,6)},\la_{(5,6,7)}}\subs\La$. Имеем $(\Om-\Om)\cap\De=\es$
и~$H:=\ha{\Om}=\bc{x\in\ha{\De}\cln x_1+x_3=x_2+x_4}\subs\ha{\De}$. Значит, $|\De\sm H|=n^2-(3^2+2^2+2^2)=32$. Кроме того,
$\La\cap H=\bc{\la_{(i_1,i_2,i_3)}\in P\cln i_1\in\{1,3\},i_2\in\{2,4\},i_3\in\{5,6,7\}}\sqcup\{\la_{(5,6,7)}\}\subs\La$, $|\La\cap H|=13$,
$|\La\sm H|=22$, $2\cdot|\La\sm H|=44>|\De\sm H|+6$.
\end{proof}

\begin{prop}\label{ths} Если $r=7$ и~$\la=\ph_3$, то существуют подмножество $\Om\subs\La$ и~гиперплоскость $H\subs\ha{\De}$, такие что $\ha{\Om}=H$,
$(\Om-\Om)\cap\De=\es$ и~$2\cdot\bgm{\Inn{\La}\sm H}>|\De\sm H|+6$.
\end{prop}

\begin{proof} Имеем $\la=\la_{(1,2,3)}$ и~$|\La|=|\De|=56$.

Положим $\Om:=\bc{\la_{(1,2,7)},\la_{(3,4,7)},\la_{(5,6,7)},\la_{(4,5,8)},\la_{(1,6,8)},\la_{(2,3,8)}}\subs\La$. Как легко заметить,
$(\Om-\Om)\cap\De=\es$ и~$H:=\ha{\Om}=\bc{x\in\ha{\De}\cln x_1+x_3+x_5=x_2+x_4+x_6}\subs\ha{\De}$. Далее,
$\La\cap H=\bc{\la_{(i_1,i_2,i_3)}\in P\cln i_1\in\{1,3,5\},i_2\in\{2,4,6\},i_3\in\{7,8\}}\subs\La$, откуда $|\La\cap H|=18$,
$|\La\sm H|=38$, $2\cdot|\La\sm H|=76>|\De|+6\ge|\De\sm H|+6$.
\end{proof}

\section{Доказательства основных результатов}\label{promain}

Данный параграф посвящён доказательству теоремы~\ref{main}.

Вернёмся к~обозначениям и~предположениям из \Ss\ref{introd}.

Положим $\de:=1\in\R$, если представление~$R$ алгебры~$\ggt_{\Cbb}$ ортогонально, и~$\de:=2\in\R$ в~противном случае.

Фиксируем максимальную коммутативную подалгебру~$\tgt$ алгебры~$\ggt$ и~картановскую подалгебру $\tgt_{\Cbb}:=\tgt\otimes\Cbb$ алгебры~$\ggt_{\Cbb}$.
В~результате возникают система корней $\De\subs\tgt_{\Cbb}^*$ и~её группа Вейля $W\subs\GL(\tgt_{\Cbb}^*)$. Имеем $\ggt\cong\sug_{r+1}$, $r>1$,
и~поэтому $\De\subs\tgt_{\Cbb}^*$\т неразложимая система корней типа~$A_r$, причём
$\ha{\De}=\bc{\la\in\tgt_{\Cbb}^*\cln\la(\tgt)\subs i\R}\subs\tgt_{\Cbb}^*=\ha{\De}\oplus i\ha{\De}$. Фиксируем систему простых корней $\Pi\subs\De$
и~соответствующую ей камеру Вейля $C\subs\ha{\De}\subs\tgt_{\Cbb}^*$.

Обозначим через $P$ и~$Q$ решётки $\bc{\la\in\ha{\De}\cln\ha{\la|\al}\in\Z\,\fa\al\in\De}\subs\ha{\De}$ и~$\ha{\De}_{\Z}\subs\ha{\De}$
соответственно. Имеем $Q\subs P$. Пусть $\la\in(P\cap C)\sm\{0\}$\т старший вес представления~$R$ алгебры~$\ggt_{\Cbb}$ относительно системы простых
корней $\Pi\subs\De\subs\tgt_{\Cbb}^*$.

Положим $\La:=W\la\subs P$, $\Inn{\La}:=\conv(\La)\cap(\La+Q)\subs P$ и~$n:=r+1\in\N$. Легко видеть, что множество весов представления~$R$
алгебры~$\ggt_{\Cbb}$ совпадает с~подмножеством $\Inn{\La}\subs P$.

Для доказательства теоремы~\ref{main} применим метод <<от противного>>\: предположим, что $V/G$\т гладкое многообразие, а~линейная алгебра
$R(\ggt_{\Cbb})$ не изоморфна ни одной из линейных алгебр $\ad(\ggt_{\Cbb})$, $\ph_1(A_r)$, $\ph_2(A_r)$, $(2\ph_1)(A_r)$ ($r>1$), $(2\ph_2)(A_3)$,
$\ph_3(A_5)$ и~$\ph_4(A_7)$. Второе условие означает, что
\eqn{\label{sug}
\begin{array}{c}
\la\notin\De\cup\{\ph_1,\ph_2,\ph_{r-1},\ph_r,2\ph_1,2\ph_r\};\\
\begin{array}{lll}
(r=3)\quad&\Ra\quad&(\la\ne2\ph_2);\\
(r=5)\quad&\Ra\quad&(\la\ne\ph_3);\\
(r=7)\quad&\Ra\quad&(\la\ne\ph_4).
\end{array}\end{array}}
Кроме того, поскольку линейные алгебры $\ph_j(A_r)$ и~$\ph_{n-j}(A_r)$ ($j=1\sco r$), а~также линейные алгебры $(\ph_1+\ph_{r-1})(A_r)$
и~$(\ph_2+\ph_r)(A_r)$ изоморфны, без ограничения общности мы можем (и~будем) считать, что
\eqn{\label{left}
\begin{array}{c}
(\la=\ph_j,\quad j=1\sco r)\quad\Ra\quad\Br{j\le\frac{n}{2}};\\
\la\ne\ph_1+\ph_{r-1}.
\end{array}}

Положим $\Pi_{\la}:=\bc{\al\in\Pi\cln\ha{\la|\al}\ne0}\subs\Pi$. Кроме того, обозначим через~$\Pc$ семейство всех неразложимых систем простых корней
$\Pi'\subs\Pi\subs\tgt_{\Cbb}^*$ порядка $r-2$.

\begin{stm}\label{cov} Предположим, что $r>2$. Тогда система простых корней~$\Pi$ совпадает с~объединением всех своих подмножеств $\Pi'\in\Pc$.
\end{stm}

\begin{proof} См. утверждение~3.1 в~\cite[\Ss3]{My0}.
\end{proof}

\begin{lemma} Пусть $\Pi'\in\Pc$\т система простых корней, удовлетворяющая соотношению $(r>2)\Ra(\Pi_{\la}\cap\Pi'\ne\es)$. Если
$\la\ne\ph_2+\ph_r$, то выполняется по крайней мере одно из нижеследующих условий \eqref{con1} и~\eqref{con2}\:
\begin{align}
&\la\in\{\ph_3\sco\ph_{r-2}\};\quad\quad\quad&&\br{\Pi'=(1\sco r-2)\subs\Pi}\ \Ra\ (r<8);\label{con1}\\
&r>2;\quad\quad\quad&&\Pi_{\la}\cap\Pi'=\{\al\}\subs\Pi,\quad\al\in\pd\Pi'\subs\Pi,\quad\ha{\la|\al}=1.\label{con2}
\end{align}
\end{lemma}

\begin{proof} Ясно, что $H:=\ba{\{\la\}\cup\Pi'}\subs\ha{\De}\subs\tgt_{\Cbb}^*$\т $(r-1)$\д мерное (вещественное) подпространство, и, следовательно,
пересечение ядер всех линейных функций этого подпространства имеет вид $\Cbb\xi\subs\tgt_{\Cbb}$, $\xi\in\tgt\sm\{0\}$.

Имеем $\Pi\cong A_r$ и~$\Pi'\cong A_{r-2}$. Как следствие, $|\De|=r(r+1)$ и~$\bgm{\De\cap\ha{\Pi'}}=(r-2)(r-1)$. Значит,
$|\De\cap H|\ge\bgm{\De\cap\ha{\Pi'}}=(r-2)(r-1)$, $|\De\sm H|\le r(r+1)-(r-2)(r-1)=4r-2$.

Допустим, что
\eqn{\label{st1}
\exi v\in V\quad\quad\quad\xi\in\ggt_v,\quad\rk\ggt_v=1.}

В~силу леммы~\ref{bms}, $\de\cdot\bgm{\Inn{\La}\sm H}\le|\De\sm H|+6\le4r+4=4n$. При этом $(\de=1)\Ra(-\La=\La)$. Мы видим, что
\equ{
\bgm{\Inn{\La}\sm H}\le4n;\quad(-\La=\La)\lor\Br{\bgm{\Inn{\La}\sm H}\le2n}.}
Далее, согласно~\eqref{left}, $\la\notin\{\ph_1+\ph_{r-1},\ph_2+\ph_r\}$. Теперь, пользуясь леммами \ref{lah} и~\ref{lah1}, а~также
соотношениями~\eqref{sug}, получаем~\eqref{con1}.

Если же условие~\eqref{st1} не выполняется, то, в~силу леммы~3.4 в~\cite[\Ss3]{My0}, имеют место соотношения~\eqref{con2}.
\end{proof}

\begin{imp} Пусть $\Pi'\in\Pc$\т система простых корней, не удовлетворяющая условию~\eqref{con1}. Если $\la\ne\ph_2+\ph_r$, то
\begin{gather}
|\Pi_{\la}\cap\Pi'|\le1,\quad\quad\Pi_{\la}\cap\br{\Pi'\sm(\pd\Pi')}=\es,\quad\quad\quad\fa\al\in\Pi_{\la}\cap\Pi'\quad\ha{\la|\al}=1;\label{undi}\\
\br{(r>2)\ \Ra\ (\Pi_{\la}\cap\Pi'\ne\es)}\quad\Ra\quad r>2.\notag
\end{gather}
\end{imp}

\begin{imp}\label{pila} Пусть $\Pi'\in\Pc$\т система простых корней, не удовлетворяющая условию~\eqref{con1}. Если $\la\ne\ph_2+\ph_r$, то $r>2$ и,
кроме того, выполняется~\eqref{undi}.
\end{imp}

Предположим, что $\la=\ph_2+\ph_r$.

В~силу~\eqref{sug}, $\la\ne2\ph_r$, откуда $r>2$. Далее, представление~$R$ алгебры~$\ggt_{\Cbb}$ не является самосопряжённым, и~поэтому $\de=2$. Теперь,
пользуясь предложением~\ref{alg} и~следствием~\ref{nosm}, получаем противоречие с~тем, что $V/G$\т гладкое многообразие.

Значит, $\la\ne\ph_2+\ph_r$. Согласно~\eqref{left}, $\la\notin\{\ph_1+\ph_{r-1},\ph_2+\ph_r\}$.

Допустим, что $\la\notin\{\ph_3\sco\ph_{r-2}\}$.

Ни одна система простых корней $\Pi'\in\Pc$ не удовлетворяет~\eqref{con1}. При этом, как легко видеть, $\Pc\ne\es$. Из следствия~\ref{pila} вытекает,
что, во-первых, $r>2$, а~во-вторых, что для любой системы простых корней $\Pi'\in\Pc$ выполняется~\eqref{undi}. Отсюда
\begin{nums}{-1}
\item для всякого $\al\in\Pi_{\la}$ имеем $\ha{\la|\al}=1$ (см. утверждение~\ref{cov})\~
\item на схеме Дынкина системы простых корней~$\Pi$ любым двум различным корням подмножества $\Pi_{\la}\subs\Pi$ соответствуют вершины, путь между
которыми включает в~себя не менее $r-2$ рёбер.
\end{nums}
Значит, $\la\in\{\ph_1\sco\ph_r\}\cup\{\ph_1+\ph_r,\ph_1+\ph_{r-1},\ph_2+\ph_r\}$ либо $r=3$ и~$\la=\ph_1+\ph_2+\ph_3$. В~то же время
$\la\notin\{\ph_3\sco\ph_{r-2}\}\cup\{\ph_1+\ph_{r-1},\ph_2+\ph_r\}$. В~силу~\eqref{sug}, $r=3$ и~$\la=\ph_1+\ph_2+\ph_3$. Представление~$R$
алгебры~$\ggt_{\Cbb}$ ортогонально, и~поэтому $\de=1$. Согласно предложению~\ref{all} и~следствию~\ref{nosm}, фактор $V/G$ не является гладким
многообразием. С~другой стороны, $V/G$\т гладкое многообразие. Получили противоречие.

Тем самым мы установили, что $\la\in\{\ph_3\sco\ph_{r-2}\}$.

Имеем $\la=\ph_j$, $j\in\N$, $3\le j\le r-2$, откуда $r\ge5$. Согласно~\eqref{left}, $j\le\frac{n}{2}$.

Если $r=5$, то $j=3$, $\la=\ph_3$, что противоречит~\eqref{sug}.

Таким образом, $r\ge6$, $3\le j\le\frac{n}{2}$ и~$\la=\ph_j$, причём, в~силу~\eqref{sug}, $(r=7)\Ra(\la\ne\ph_4)$.

\begin{cas} $r=6$ и~$\la=\ph_3$.
\end{cas}

Представление~$R$ алгебры~$\ggt_{\Cbb}$ не является самосопряжённым. Отсюда $\de=2$. Применяя предложение~\ref{thi} и~следствие~\ref{nosm}, получаем
противоречие с~тем, что $V/G$\т гладкое многообразие.

\begin{cas} $r=7$ и~$\la=\ph_3$.
\end{cas}

Представление~$R$ алгебры~$\ggt_{\Cbb}$ не является самосопряжённым. Значит, $\de=2$. Пользуясь предложением~\ref{ths} и~следствием~\ref{nosm}, приходим
к~противоречию с~тем, что $V/G$\т гладкое многообразие.

\begin{cas} $r\ge8$.
\end{cas}

Система простых корней $\Pi':=(1\sco r-2)\subs\Pi$, принадлежащая семейству~$\Pc$, не удовлетворяет условию~\eqref{con1} и, в~силу следствия~\ref{pila},
удовлетворяет~\eqref{undi}. В~частности, $\Pi_{\la}\cap\br{\Pi'\sm(\pd\Pi')}=\es$. При этом $\Pi'\cong A_{r-2}$ и~$r-2\ge6$, откуда
$\pd\Pi'=(1,2,r-3,r-2)\subs\Pi$, $\Pi'\sm(\pd\Pi')=(3\sco r-4)\subs\Pi$. Далее, $j\le\frac{n}{2}=n-\frac{n}{2}<n-\frac{r}{2}\le n-4=r-3$,
$3\le j\le r-4$, а~также $\Pi_{\la}=(j)\subs\Pi$. Значит, $\Pi_{\la}\cap\br{\Pi'\sm(\pd\Pi')}=(j)\ne\es$. Получили противоречие.

Тем самым мы полностью доказали (методом <<от противного>>) теорему~\ref{main}.

\section{Разбор частных случаев}\label{promain1}

В~этом параграфе будет доказана теорема~\ref{main1}.

Как и прежде, мы будем пользоваться обозначениями и~предположениями из \Ss\ref{introd}.

Положим $n:=r+1\in\N$.

Допустим, что линейная группа Ли $G\subs\GL(V)$ связна, а~линейная алгебра $R(\ggt_{\Cbb})$ изоморфна одной из следующих линейных алгебр\:
\begin{nums}{-1}
\item\label{adj} $\ad(\ggt_{\Cbb})$\~
\item\label{tav} $\ph_1(A_r)$ \ter{$r>1$}\~
\item\label{sym} $(2\ph_1)(A_r)$ \ter{$r>1$}\~
\item\label{ext} $\ph_2(A_r)$ \ter{$r>3$}\~
\item\label{ort} $\ph_2(A_3)\cong\ph_1(D_3)$\~
\item\label{ord} $(2\ph_2)(A_3)\cong(2\ph_1)(D_3)$\~
\item\label{thd} $\ph_3(A_5)$\~
\item\label{fth} $\ph_4(A_7)$.
\end{nums}

Вначале разберём случаи \ref{adj}, \ref{tav}, \ref{ort}, \ref{ord} и~\ref{fth}.

Следующие представления комплексных простых групп Ли являются полярными (см.~\cite[\Ss3]{CD})\:
\begin{itemize}
\item присоединённое представление произвольной комплексной простой группы Ли\~
\item представление $R_{\ph_1}+R'_{\ph_1}$ комплексной простой группы Ли $\SL_n(\Cbb)$\~
\item представления $R_{\ph_1}$ и~$R_{2\ph_1}$ комплексной простой группы Ли $\SO_6(\Cbb)$\~
\item представление~$R_{\ph_4}$ комплексной простой группы Ли $\SL_8(\Cbb)$.
\end{itemize}
Значит, в~каждом из случаев \ref{adj}, \ref{tav}, \ref{ort}, \ref{ord} и~\ref{fth} линейная группа Ли $G\subs\GL(V)$ является полярной и, в~силу
леммы~\ref{pol}, фактор $V/G$ гомеоморфен замкнутому полупространству (в~частности, не является многообразием).

Теперь разберём случай~\ref{thd}.

Допустим, что линейная группа $G\subs\GL(V)$ связна, а~линейная алгебра $R(\ggt_{\Cbb})$ изоморфна линейной алгебре $\ph_3(A_5)$.

Рассмотрим $n$\д мерное эрмитово пространство~$\Eb$. Это пространство допускает разложение в~прямую сумму двух ортогональных трёхмерных подпространств
$\Eb_+$ и~$\Eb_-$.

Ясно, что тавтологическое представление $G\cln V$ изоморфно естественному действию $\SU(\Eb)\cln\Eb^{\wg3}$. Без ограничения общности будем считать,
что $G=\SU(\Eb)$, $V=\Eb^{\wg3}$, а~тавтологическое представление $G\cln V$ совпадает с~естественным действием $\SU(\Eb)\cln\Eb^{\wg3}$.

Как легко заметить, пространство~$V$ разлагается в~прямую сумму своих подпространств $V_j:=\Eb_+^{\wg(3-j)}\wg\Eb_-^{\wg j}$ ($j=0,1,2,3$), причём
$V_0=\Cbb v$, $v\in V_0\sm\{0\}$. Далее, имеем $G_v=\SU(\Eb_+)\times\SU(\Eb_-)\subs G$, где $\SU(\Eb_{\pm}):=\{g\in G\cln\Eb^g\sups\Eb_{\mp}\}\subs G$.
Ввиду вышесказанного, $G_vV_j=V_j$ ($j=0,1,2,3$), $\ggt v=(i\R v)\oplus V_1$, $V=(\ggt v)\oplus(\R v)\oplus V_2\oplus V_3$, $V_3\subs V^{G_v}$,
$\dim_{\Cbb}V_3=1$. Поэтому представление $G_v\cln N_v$ изоморфно прямой сумме тождественного действия $G_v\cln\R^3$ и~представления $G_v\cln V_2$.
Последнее, в~свою очередь, изоморфно естественному действию $\br{\SU(\Eb_+)\times\SU(\Eb_-)}\cln(\Eb_+\otimes\Eb_-^{\wg2})$ и, что равносильно,
естественному действию $\br{\SU(\Eb_+)\times\SU(\Eb_-)}\cln(\Eb_+\otimes\Eb_-^*)$. Теперь, полагая $d:=m:=3\in\N$ и~применяя к~представлению
$G_v\cln N_v$ утверждение~\ref{read}, получаем, что фактор данного представления не является многообразием. Согласно лемме~\ref{slice}, не является им
и~фактор $V/G$.

Перейдём к~разбору оставшихся случаев \ref{sym} и~\ref{ext}.

Далее будем считать, что линейная группа $G\subs\GL(V)$ связна, а~линейная алгебра $R(\ggt_{\Cbb})$ изоморфна одной из линейных алгебр
$(2\ph_1)(A_r)$ \ter{$r>1$} и~$\ph_2(A_r)$ \ter{$r>3$}.

Положим $\de:=1\in\R$ (соотв. $\de:=-1\in\R$), если линейная алгебра $R(\ggt_{\Cbb})$ изоморфна линейной алгебре $(2\ph_1)(A_r)$ (соотв. линейной
алгебре $\ph_2(A_r)$).

Допустим, что $\de=-1$ и~$n\in2\Z+1$.

Представление $R_{\ph_2}+R'_{\ph_2}$ комплексной простой группы Ли $\SL_n(\Cbb)$ является полярным (см.~\cite[\Ss3]{CD}). Согласно лемме~\ref{pol},
фактор $V/G$ гомеоморфен замкнутому полупространству и, как следствие, не является многообразием.

С~этого момента будем считать, что $(\de=1)\lor(n\in2\Z)$.

Рассмотрим $n$\д мерное эрмитово пространство~$\Eb$ со скалярным умножением $f_0(\cdot,\cdot)$, комплексное пространство~$B$ всех билинейных форм на
пространстве~$\Eb$ и~подпространство $B_{\de}\subs B$ всех форм $f\in B$, таких что $f(x,y)=\de\cdot f(y,x)$ ($x,y\in\Eb$).

Тавтологическое представление $G\cln V$ изоморфно представлению
\eqn{\label{repr}
\begin{array}{c}
\SU(\Eb)\cln B_{\de},\quad\quad(gf)(x,y):=f(g^{-1}x,g^{-1}y)\in\Cbb\\
(g\in\SU(\Eb),\ \ f\in B_{\de},\ \ x,y\in\Eb).
\end{array}}

Найдётся форма $f_1\in B_{\de}$, для которой (фиксированный) ортонормированный базис эрмитова пространства~$\Eb$ является ортонормированным при $\de=1$
и~симплектическим при $\de=-1$. При этом, очевидно, существует антилинейный оператор $F\cln\Eb\to\Eb$, такой что $F^2=\de E\in\Un(\Eb)$
и~$f_1(x,y)=f_0(Fx,y)$ ($x,y\in\Eb$).

Пусть $\ta$, $\ta_0$ и~$\ta_1$\т инволютивные автоморфизмы \textit{вещественной} группы Ли $\GL_{\Cbb}(\Eb)$, однозначно задаваемые соотношениями
$\ta(g)F=Fg$ и~$f_i\br{\ta_i(g)x,gy}=f_i(x,y)$ ($i=0,1$, $g\in\GL_{\Cbb}(\Eb)$, $x,y\in\Eb$). Для произвольных $g\in\GL_{\Cbb}(\Eb)$ и~$x,y\in\Eb$ имеем
\equ{
f_0\br{F\ta_1(g)x,gy}=f_1\br{\ta_1(g)x,gy}=f_1(x,y)=f_0(Fx,y)=f_0\br{\ta_0(g)Fx,gy},}
откуда $F\ta_1(g)=\ta_0(g)F$. Следовательно, $\ta_0\eqi\ta\circ\ta_1$, $\ta_0\circ\ta_1\eqi\ta_1\circ\ta_0\eqi\ta$. Вещественная алгебра Ли
$\Lie\br{\GL_{\Cbb}(\Eb)}=\glg_{\Cbb}(\Eb)$ обладает коммутирующими между собой инволютивными автоморфизмами $\si_i:=d\ta_i$ ($i=0,1$). Легко видеть, что
$f_i\br{\si_i(\xi)x,y}+f_i(x,\xi y)=0$ ($i=0,1$, $\xi\in\glg_{\Cbb}(\Eb)$, $x,y\in\Eb$). Значит, автоморфизм~$\si_0$ (соотв. автоморфизм~$\si_1$)
вещественной алгебры Ли $\glg_{\Cbb}(\Eb)$ антилинеен (соотв. линеен) над полем~$\Cbb$, а~подпространство
$\br{\glg_{\Cbb}(\Eb)}^{-\si_1}\subs\glg_{\Cbb}(\Eb)$ содержит подпространство $\Cbb E\subs\glg_{\Cbb}(\Eb)$.

Ясно, что $\br{\GL_{\Cbb}(\Eb)}^{\ta_0}=\Un(\Eb)\subs\GL_{\Cbb}(\Eb)$. Далее, отображение
\equ{
S\cln\glg_{\Cbb}(\Eb)\to B,\quad\quad\br{S(\xi)}(x,y):=f_1(x,\xi y)\in\Cbb\quad\quad(\xi\in\glg_{\Cbb}(\Eb),\ \ x,y\in\Eb)}
является изоморфизмом комплексных линейных пространств. Кроме того, для любых $g\in\SU(\Eb)$, $\xi\in\glg_{\Cbb}(\Eb)$ и~$x,y\in\Eb$ имеем
\begin{gather*}
\br{S(\xi)}(g^{-1}x,g^{-1}y)=f_1(g^{-1}x,\xi g^{-1}y)=f_1\br{x,\ta_1(g)\xi g^{-1}y}=\Br{S\br{\ta_1(g)\xi g^{-1}}}(x,y);\\
\de\cdot\br{S(\xi)}(y,x)=\de\cdot f_1(y,\xi x)=f_1(\xi x,y)=-f_1\br{x,\si_1(\xi)y}=-\Br{S\br{\si_1(\xi)}}(x,y).
\end{gather*}
Отсюда следует, что $S^{-1}(B_{\de})=\br{\glg_{\Cbb}(\Eb)}^{-\si_1}\subs\glg_{\Cbb}(\Eb)$, причём представление~\eqref{repr} изоморфно представлению
\eqn{\label{reps}
\SU(\Eb)\cln\br{\glg_{\Cbb}(\Eb)}^{-\si_1},\quad g\cln\xi\to\ta_1(g)\xi g^{-1}.}
Как уже было сказано, тавтологическое представление $G\cln V$ изоморфно представлению~\eqref{repr}, а~значит, и~представлению~\eqref{reps}.
Без ограничения общности будем считать, что $G=\SU(\Eb)\subs\GL_{\Cbb}(\Eb)$, $V=\br{\glg_{\Cbb}(\Eb)}^{-\si_1}\subs\glg_{\Cbb}(\Eb)$, а~тавтологическое
представление $G\cln V$ совпадает с~представлением~\eqref{reps}.

Коммутирующие между собой автоморфизмы $\ta_0$ и~$\ta_1$ группы Ли $\GL_{\Cbb}(\Eb)$ переводят в~себя её коммутант $\SL_{\Cbb}(\Eb)$, а~значит,
и~подгруппу $\br{\SL_{\Cbb}(\Eb)}^{\ta_0}=\SL_{\Cbb}(\Eb)\cap\Un(\Eb)=G$.

Коммутирующие между собой инволютивные автоморфизмы $\si_0$ и~$\si_1$ алгебры $\glg_{\Cbb}(\Eb)$ переводят в~себя её центр $\Cbb E$ и~коммутант
$\slg_{\Cbb}(\Eb)$, что позволяет разложить эту алгебру в~прямую сумму $\si_1$\д инвариантных подпространств $\Cbb E$,
$\br{\slg_{\Cbb}(\Eb)}^{\si_0}$ и~$\br{\slg_{\Cbb}(\Eb)}^{-\si_0}$. Далее, из равенства $\br{\SL_{\Cbb}(\Eb)}^{\ta_0}=G\subs\GL_{\Cbb}(\Eb)$ следует, что
$\br{\slg_{\Cbb}(\Eb)}^{\si_0}=\ggt\subs\glg_{\Cbb}(\Eb)$. Поскольку $\Cbb\br{\slg_{\Cbb}(\Eb)}=\slg_{\Cbb}(\Eb)\subs\glg_{\Cbb}(\Eb)$, а~$\si_0$\т
антилинейный над полем~$\Cbb$ автоморфизм алгебры $\glg_{\Cbb}(\Eb)$, имеем
$\br{\slg_{\Cbb}(\Eb)}^{-\si_0}=i\cdot\br{\slg_{\Cbb}(\Eb)}^{\si_0}=i\ggt\subs\glg_{\Cbb}(\Eb)$.

Согласно вышесказанному, $V=\br{\glg_{\Cbb}(\Eb)}^{-\si_1}=(\Cbb E)^{-\si_1}\oplus\ggt^{-\si_1}\oplus(i\ggt)^{-\si_1}\subs\glg_{\Cbb}(\Eb)$ и,
кроме того, $(\si_1-\id_{\ggt})\ggt=\ggt^{-\si_1}\subs\ggt$. Пользуясь соотношениями $(\Cbb E)^{-\si_1}=\Cbb E\subs\glg_{\Cbb}(\Eb)$
и~$\si_1\in\Aut_{\Cbb}\br{\glg_{\Cbb}(\Eb)}$, получаем, что
\eqn{\label{dis}
V=(\Cbb E)\oplus\ggt^{-\si_1}\oplus i(\ggt^{-\si_1})\subs\glg_{\Cbb}(\Eb).}

Ясно, что $v:=E\in V$. При этом (см.~\eqref{reps}) $G_v=G^{\ta_1}\subs G$, $\ggt v=(\si_1-\id_{\ggt})\ggt=\ggt^{-\si_1}\subs\ggt$, а~действие
$G_v\cln V$ осуществляется по правилу $g\cln\xi\to g\xi g^{-1}$. В~силу~\eqref{dis}, представление $G_v\cln N_v$ изоморфно прямой сумме
тождественного действия $G^{\ta_1}\cln\R^2$ и~присоединённого действия $G^{\ta_1}\cln\ggt^{-\si_1}$.

Согласно лемме~\ref{wey}, фактор $\ggt^{-\si_1}/G^{\ta_1}$ гомеоморфен замкнутому полупространству. Кроме того,
$N_v/G_v\cong\R^2\times(\ggt^{-\si_1}/G^{\ta_1})$ и, следовательно, фактор $N_v/G_v$ также гомеоморфен замкнутому полупространству.

Тем самым мы установили, что фактор $N_v/G_v$ не является многообразием. Согласно лемме~\ref{slice}, не является им и~фактор $V/G$.

Таким образом, в~каждом из случаев \ref{sym} и~\ref{ext} фактор $V/G$ не является многообразием.

Теперь теорема~\ref{main1} полностью доказана.

Из теорем \ref{main} и~\ref{main1} сразу вытекает следствие~\ref{main2}.

\end{document}